\documentclass[reqno, 12pt]{amsart}
\pdfoutput=1
\makeatletter
\let\origsection=\section \def\section{\@ifstar{\origsection*}{\mysection}}
\def\mysection{\@startsection{section}{1}\z@{.7\linespacing\@plus\linespacing}{.5\linespacing}{\normalfont\scshape\centering\S}}
\makeatother

\usepackage{amsmath,amssymb,amsthm}
\usepackage{mathrsfs}
\usepackage{mathabx}\changenotsign
\usepackage{dsfont}

\usepackage[ruled,vlined]{algorithm2e}
\usepackage{setspace}
\usepackage{ragged2e}
\usepackage{xpatch}

\makeatletter
\renewcommand{\Indentp}[1]{%
  \advance\leftskip by #1
  \advance\skiptext by -#1
  \advance\skiprule by #1}%
\renewcommand{\Indp}{\algocf@adjustskipindent\Indentp{12pt}}
\makeatother

\usepackage[normalem]{ulem}

\usepackage{graphicx}
\usepackage{float}

\usepackage{xcolor}
\usepackage{hyperref}
\hypersetup{
	colorlinks,
	linkcolor={red!60!black},
	citecolor={green!60!black},
	urlcolor={blue!60!black}
}

\definecolor{codelightgray}{gray}{0.8}
\definecolor{codeverylightgray}{gray}{0.9}

\usepackage{array,multirow,colortbl,enumerate}
\usepackage{caption} 

\usepackage{graphicx}

\usepackage[open,openlevel=2,atend]{bookmark}

\usepackage[abbrev,msc-links]{amsrefs}
\usepackage{amsrefs}
\usepackage{doi}

\renewcommand{\PrintDOI}[1]{\doi{#1}}

\usepackage[T1]{fontenc}
\usepackage{lmodern}
\usepackage[babel]{microtype}
\usepackage[english]{babel}

\usepackage{wasysym}

\usepackage{geometry}
\numberwithin{equation}{section}
\numberwithin{figure}{section}

\usepackage{enumerate}
\usepackage{enumitem}

\let\polishlcross=\l
\def\l{\ifmmode\ell\else\polishlcross\fi}

\def\argmin{\text{argmin}}

\def\paragraph#1{%
	\noindent\textbf{#1.}\enspace}

\let\emptyset=\varnothing
\let\setminus=\smallsetminus

\makeatletter
\def\moverlay{\mathpalette\mov@rlay}
\def\mov@rlay#1#2{\leavevmode\vtop{   \baselineskip\z@skip \lineskiplimit-\maxdimen
		\ialign{\hfil$\m@th#1##$\hfil\cr#2\crcr}}}
\newcommand{\charfusion}[3][\mathord]{
	#1{\ifx#1\mathop\vphantom{#2}\fi
		\mathpalette\mov@rlay{#2\cr#3}
	}
	\ifx#1\mathop\expandafter\displaylimits\fi}
\makeatother

\DeclareFontFamily{U}  {MnSymbolC}{}
\DeclareSymbolFont{MnSyC}         {U}  {MnSymbolC}{m}{n}
\DeclareFontShape{U}{MnSymbolC}{m}{n}{
	<-6>  MnSymbolC5
	<6-7>  MnSymbolC6
	<7-8>  MnSymbolC7
	<8-9>  MnSymbolC8
	<9-10> MnSymbolC9
	<10-12> MnSymbolC10
	<12->   MnSymbolC12}{}
\DeclareMathSymbol{\powerset}{\mathord}{MnSyC}{180}

\usepackage{tikz}
\newcommand*\circled[1]{\tikz[baseline=(char.base)]{
            \node[shape=circle,draw,inner sep=2pt] (char) {#1};}}
\usetikzlibrary{calc,decorations.pathmorphing,decorations.pathreplacing,shapes.geometric, arrows}
\pgfdeclarelayer{background}
\pgfdeclarelayer{foreground}
\pgfdeclarelayer{front}
\pgfsetlayers{background,main,foreground,front}

\usepackage{thmtools}
\usepackage{thm-restate}
\usepackage{cleveref}

\let\epsilon=\varepsilon
\let\eps=\epsilon
\let\rho=\varrho
\let\theta=\vartheta
\let\kappa=\varkappa

\let\E=\EE
\def\NN{{\mathds N}}

\let\Prob=\PP

\newcommand{\cP}{\mathcal{P}}

\newcommand{\ham}{\mathcal{HAM}}

\theoremstyle{plain}
\newtheorem{thm}{Theorem}[section]
\newtheorem{theorem}[thm]{Theorem}

\newtheorem{prop}[thm]{Proposition}

\newtheorem{fact}[thm]{Fact}
\newtheorem{cor}[thm]{Corollary}
\newtheorem{lemma}[thm]{Lemma}

\newtheorem{obs}[thm]{Observation}

\theoremstyle{definition}
\newtheorem{rem}[thm]{Remark}
\newtheorem{dfn}[thm]{Definition}

\newtheorem{conj}[thm]{Conjecture}

\usepackage{accents}
\newcommand{\seq}[1]{\accentset{\rightharpoonup}{#1}}

\let\phi=\varphi

\begin{document}
	
	\title[Powers of Hamiltonian cycles in randomly augmented graphs]{Powers of Hamiltonian cycles in randomly augmented P\'osa-Seymour  graphs }
	
	\author[S. Antoniuk]{Sylwia Antoniuk}
	\address{Department of Discrete Mathematics, Adam Mickiewicz University, Pozna\'n, Poland}
	\email{antoniuk@amu.edu.pl}
	\thanks{The first author was supported by Narodowe Centrum Nauki, grant 2024/53/B/ST1/00164}

	\author[A. Dudek]{Andrzej Dudek}
	\address{Department of Mathematics, Western Michigan University, Kalamazoo, MI, USA}
	\email{andrzej.dudek@wmich.edu}
	\thanks{The second author was supported in part by Simons Foundation Grant MPS-TSM-00007551.}

	\author[A. Ruci\'nski]{Andrzej Ruci\'nski}
	\address{Department of Discrete Mathematics, Adam Mickiewicz University, Pozna\'n, Poland}
	\email{\tt rucinski@amu.edu.pl}
	\thanks{The third author was supported by Narodowe Centrum Nauki, grant 2024/53/B/ST1/00164}

	\begin{abstract}
 We study the question of the least number of  random edges that need to be added to a P\'osa-Seymour  graph, that is, a graph with  minimum degree exceeding $\frac k{k+1}n$, to secure the existence of the $m$-th power of a Hamiltonian cycle, $m>k$. It turns out that, depending on $k$ and $m$, this quantity may be captured by two types of thresholds, with one of them, called over-threshold, becoming dominant for large $m$. Indeed, for each $k\ge2$ and $m>m_0(k)$, we establish asymptotically tight lower and upper bounds on the over-thresholds (provided they exist) and show that for infinitely many instances of $m$ the two bounds coincide.
  In addition, we also determine the thresholds for some small values of $k$ and $m$.
	\end{abstract}
	
	\maketitle


	
\section{Introduction}

The study of \emph{randomly augmented} graphs, also called \emph{randomly perturbed graphs}, was initiated by Bohman, Frieze, and Martin in \cite{BFM2003}. The general problem can be described as follows: given a family $\mathcal G$ of $n$-vertex graphs and a graph property~$\mathcal{P}$, determine minimal $p=p(n)$ which ensures that for \emph{every} $G\in\mathcal G$ \emph{asymptotically almost surely (briefly: a.a.s.)} $G\cup G(n,p)\in \mathcal{P}$. Here $G(n,p)$ is the standard random binomial graph with $p=p(n)$. In \cite{BFM2003}, the authors considered the case where, given $0<\epsilon<1/2$, $\mathcal G$ is the family of all $n$-vertex graphs $G$ with minimum degree $\delta(G)\geq \eps n$, while $\cP$ is the Hamiltonicity, and showed that the threshold sequence $p(n)$ for $G \cup G(n,p)$ being Hamiltonian is of order $1/n$.

Several papers (see e.g.~\cite{DRRS}, \cite{ADRRS}, \cite{BPSS2022}, \cite{ADR}) followed that suit and under other Dirac-type conditions considered higher powers of Hamiltonian cycles.

For $m\in\NN$ the \emph{$m$-th power~$F^m$} of a~graph~$F$ is defined as the graph on the same vertex set as $F$, and whose edges join distinct vertices at distance at most~$m$ in~$F$.
The $m$-th power of an $n$-vertex path/cycle will be often called an \emph{$m$-path}/\emph{$m$-cycle} and denoted by $P_n^m$ and $C_n^m$, respectively.
A~Hamiltonian cycle in a graph $G$ is a cycle which passes through all vertices of~$G$. For integers $m\ge1$ and $n\geq m+2$, the family $\ham_n^m$ consists of all $n$-vertex graphs~$G$ that contain the $m$-th power of a Hamiltonian cycle.

Recall that the celebrated P\'osa-Seymour conjecture asserted that every graph $G$ with $n$ vertices and minimum degree $\delta(G)\ge\tfrac k{k+1}n$ contains the $k$-th power of a Hamiltonian cycle.  Koml\'os, S\'ark\"ozy, and Szemer\'edi~\cite{KSS1996,KSS1998} proved this conjecture for $n$ large enough.   The deterministic graphs we are going to consider will always satisfy a slightly stronger condition $\delta(G)\ge\left(\tfrac k{k+1}+\eps\right)n$ and we will be interested in probability $p(n)$ guaranteeing higher (than $k$) powers of Hamilton cycles in $G\cup G(n,p)$.
	

\subsection{Thresholds and over-thresholds: known results}\label{known} As customary, any definition of a threshold consists of a 1-statement and a 0-statement, with 1 and 0 indicating the respective limiting probability. This is also the case of thresholds for properties of randomly augmented graphs.

\begin{dfn}\label{thres}
We say that a function $d(n)$ is a \emph{$(k,m)$-Dirac threshold} if
    \begin{itemize}
	\item[(i)] for every $\eps>0$ there exists     $c_1>0$ such that for all $n$-vertex graphs $G$ with $\delta(G)\ge\left(\tfrac k{k+1}+\eps\right)n$ and all $p:=p(n)\ge c_1d(n)$
    $$\lim_{n\to\infty}\Prob\left(G\cup G(n,p)\in\ham^m_n\right)=1,$$
	and
	\item[(ii)]  there exist $\epsilon>0$, $c_0>0$, and a sequence of $n$-vertex graphs $G_\epsilon:=G_\epsilon(n)$ with $\delta(G_\eps)\ge\left(\tfrac k{k+1}+\epsilon\right)n$ such that for every $p:=p(n)\le c_0d(n)$
    $$\lim_{n\to\infty}\Prob\left(G_\eps\cup G(n,p)\in\ham^m_n\right)=0.$$
    \end{itemize}

If exists, the $(k,m)$-Dirac threshold is denoted by $d_{k,m}(n)$. In such case, the 1-statement alone yields the upper bound $d_{k,m}(n)\le d(n)$, while the 0-statement alone yields the lower bound $d_{k,m}(n)\ge d(n)$.
\end{dfn}

Note that, by definition, the above threshold $d(n)$ does not depend on $\eps$ (only the constants $c_1$ and $c_0$ do).
	
In \cite{DRRS} it was proved that $d_{k,k+1}(n)=n^{-1}$ for all $k\ge1$. This was substantially extended in \cite{ADRRS} to cover many other $(k,m)$-Dirac thresholds. In particular, it turned out that $d_{k,m}(n)=n^{-1}$ for all $m\le 2k+1$. This result was independently obtained by Nenadov and Truji\'{c} in~\cite{NT}.

We now state the main result of \cite{ADRRS}. For the ease of future references, we split it into the 1-statement, the 0-statement, and a unifying conclusion.

\begin{theorem}[\cite{ADRRS}]\label{21}
Let $k\in\mathbb{N}$.
\begin{enumerate}
    \item[(a)] For all integers $r\geq 0$, $\ell\geq r(r+1)$, and $m\le k\ell+r$, if $d_{k, m}(n)$ exists, then $d_{k, m}(n)\leq n^{-2/\ell}$.
    \item[(b)] For all positive integers $\ell$ and $m\ge (k+1)(\ell-1)$, if $d_{k, m}(n)$ exists, then $d_{k,m}(n)\geq n^{-2/\ell}$.
    \item[(c)] Consequently, for $(k+1)(\ell-1)\le m\le k\ell+r$ and $\ell\geq r(r+1)$, we have $d_{k,m}(n)= n^{-2/\ell}$.
  \end{enumerate}
\end{theorem}

The $(k,m)$-Dirac thresholds determined so far are of the form $d_{k,m}(n)=n^{-\eta_{k,m}}$, with a~positive rational $\eta_{k,m}\le1$. Let us call this number the \emph{$(k,m)$-Dirac exponent}.

Despite the apparent generality of Theorem~\ref{21}, a vast majority of pairs $(k,m)$ still remained unaccounted for. For instance, for $k=1$ Theorem~\ref{21}(c) covers only $m\in\{2,3,4\}$, for $k=2$ only $m\in\{3,4,5,6,7,9\}$, while for $k=3$ only $m\in\{4,\dots,10,12,13,16,20\}$. By some ad hoc arguments, the exponents $\eta_{1,8}$ and $\eta_{2,14}$ have also been determined in \cite{ADRRS}.

Recently, in \cite{ADR} we have fully covered the case $k=1$, showing that except for a few small values of $m$, the thresholds change their nature and become more elusive, which lead us to the following definition of an \emph{Dirac over-threshold.}
	
\begin{dfn}\label{over} We say that a positive rational $\eta$ is a \emph{$(k,m)$-Dirac over-exponent} if
    \begin{itemize}
        \item[(i)] for every $\eps>0$ there exists $\mu>0$ such that for all $n$-vertex graphs $G$ with $\delta(G)\ge\left(\tfrac k{k+1}+\eps\right)n$ and all $p:=p(n)\ge n^{-\eta-\mu}$
        $$\lim_{n\to\infty}\Prob\left(G\cup G(n,p)\in\ham^m_n\right)=1,$$
        and
        \item[(ii)]  for every real $\mu>0$ there exists $\epsilon>0$ and a sequence of $n$-vertex graphs $G_\epsilon:=G_\epsilon(n)$ with $\delta(G_\eps)\ge\left(\tfrac k{k+1}+\epsilon\right)n$ such that for every $p:=p(n)\le n^{-\eta-\mu}$
	$$\lim_{n\to\infty}\Prob\left(G_\eps\cup    G(n,p)\in\ham^m_n\right)=0.$$
    \end{itemize}
If exists, the $(k,m)$-Dirac over-exponent is denoted by $\bar \eta_{k,m}$ and the function $\bar d_{k,m}(n) := n^{-\bar\eta_{k,m}}$ is called the \emph{$(k,m)$-Dirac over-threshold}. Then, the 1-statement alone yields the bound $\bar \eta_{k,m}\ge \eta$, equivalently $\bar d_{k,m}(n)\le n^{-\eta}$, while the 0-statement alone yields the bound $\bar \eta_{k,m}\le \eta$, equivalently $\bar d_{k,m}(n)\ge n^{-\eta}$.

Throughout the paper, whenever we write any of these inequalities, we assume implicitly that the object in question, that is $\eta_{k,m}$, $\bar \eta_{k,m}$, $d_{k,m}(n)$, or $\bar d_{k,m}(n)$, exists.
\end{dfn}

The prefix ``over'' is meant to remind us that the abrupt change in behavior of the probability in question happens just ``below'' the function $n^{-\eta_{k,m}}$. The main difference between the $(k,m)$-Dirac threshold  and the $(k,m)$-Dirac over-threshold  is that now the implicit dependence on $\epsilon$ is much more substantial. As a drawback, however, for a given $\eps$ we do not obtain a~threshold in the classical sense, but a pair of functions bounding it from both sides.

More precisely, for every $\eps>0$ the 1-statement (i) holds for $p\ge n^{-\eta-\mu_1(\eps)}$ while the 0-statement (ii) holds for $p\le n^{-\eta-\mu_2(\eps)}$, with $\mu_1(\eps)<\mu_2(\eps)$, where $\mu_2(\eps)$ is the inverse of the function $\eps(\mu)$ defined in part (ii) of Definition~\ref{over}. In all known cases, it follows from the proofs that $\mu_2(\eps)$ is a linear function of $\eps$, while $\mu_1(\eps)$ is a tower function of $\eps$ born in the Regularity Lemma.

\medskip
Throughout the paper, we will be using notation $\argmin f(x)[D]$ to denote the (smallest) element $x_0\in D$ for which $f(x_0)=\min_{x\in D}f(x)$.
To formulate the result for $k=1$ from \cite{ADR} in full generality, for each integer $m\ge2$, let $f:=f_m$  be a function with the real domain $(0,m)$ defined as
$$f(x)=\frac1{x}\left(\binom x2+\binom{m-x+1}{2}\right).$$
Observe that $f$ is a convex function. It can be easily checked that $f$ has a unique global minimum at
$$\lambda_m:=\argmin f(x)[(0,m)]= \frac{\sqrt{2m^2+2m}}{2}.$$
Moreover, owing to the convexity of $f$,
$$\ell_m:=\argmin f(x)[[m]]\in\{\lfloor\lambda_m\rfloor, \lceil\lambda_m\rceil\}.$$
For example, $\lambda_{20}=\sqrt{210}$, $\lfloor\lambda_{20}\rfloor = 14$, $\lceil\lambda_{20}\rceil = 15$, and $f(14)=8=f(15)$, so $\ell_{20} = 14$.

\begin{theorem}[\cite{ADR}]\label{23}~
\begin{enumerate}
    \item[(a)]	For all integers $r\ge0$, $r\le \ell<r(r+1)$ and $m=\ell+r$,  we have $\bar\eta_{1,m}\ge 1/f(\ell)$.
    \item[(b)]
    For every integer $m\geq 2$, we have $\bar\eta_{1,m}\le 1/f(\ell_m)$.
    \item[(c)] For $m=7$ and all $m\geq 10$, setting $r=m-\ell_m$, we have $r\le \ell_m<r(r+1)$. Consequently, $\bar \eta_{1,m} = 1/f(\ell_m).$
\end{enumerate}
\end{theorem}
	
Recall that in \cite{DRRS,ADRRS} we have already found the $(1,m)$-Dirac exponents for $m\in\{2,3,4,8\}$. Also in \cite{ADRRS}, by an ad hoc argument, we have determined the $(1,5)$-Dirac over-exponent. The  remaining two cases, $m=6$ and $m=9$, were established separately in \cite{ADR}. These three additional cases are exceptional in that the over-exponents are not equal to $1/f(\ell_m)$ (see discussion in Section~\ref{extinct}).
Anyhow, in \cite{ADR} we have obtained the complete collection of $(1,m)$-Dirac thresholds and over-thresholds. They are summarized in Table~\ref{table:exponents}.
\begin{table}
    \renewcommand{\arraystretch}{1.3}
    \begin{tabular}{ | >{\columncolor{codelightgray}}c || c | c |c | c | c | c | c | c | c | c |}
	\hline
	\rowcolor{codelightgray}
	$m$ & 2\,\cite{DRRS} & 3\,\cite{ADRRS, NT} &         4\,\cite{ADRRS} & 5\,\cite{ADRRS} &                  6\,\cite{ADR} & 7\,\cite{ADR} & 8\,\cite{ADRRS} &    9\,\cite{ADR} & $\ge10$\,\cite{ADR}\\                \hline\hline
	$1/\eta_{1,m}$ & 1 & 1 & $\frac{3}{2}$ & - & - &      - & $3$ & - & -\\ \hline\hline
	$1/\bar\eta_{1,m}$ & - & - &- & $2$ & $\frac{9}      {4}$ & $\frac{13}{5}$ & - & $\frac{7}{2}$ &          $f(\ell_m)$\\ \hline
    \end{tabular}
    \caption{The  reciprocals of the $(1,m)$-Dirac exponents and over-exponents.}
    \label{table:exponents}
\end{table}

	
\subsection{New results}\label{newres}

As it will be explained in the next section, the proof of Theorem~\ref{21}(a), under the complementary assumption $\ell<r(r+1)$, can be easily adapted to yield a~generalization of Theorem~\ref{23}(a) to larger $k$.
For all $k\ge1$ and $m\ge k+1$, define a function $f=f_{k,m}$ over the real domain $(0,m)$ as
\[ f(x)=\frac1{x}\left(\binom x2+\binom{m-kx+1}{2}\right).\]
Note that $f_{1,m}=f_m$ was defined prior to Theorem~\ref{23} and that, in fact, setting $m=k\ell+r$ and $x=\ell$, all functions $f_{k,m}$ take the same form
$$f(\ell)=\frac1{\ell}\left(\binom \ell2+\binom{r+1}{2}\right).$$
Next, observe that
\begin{align*}
f(x) &= \frac12\left( x-1 + \frac{m(m+1)}{x} + k^2x - 2km - k \right)\\	
&= \frac12\left(\sqrt{x(k^2+1)} - \sqrt{\frac{m(m+1)}{x}} \right)^2 + \sqrt{(k^2+1)(m^2+m)} - \frac{(2m+1)k}{2} - \frac12.
\end{align*}
Thus, $f$ has a unique global minimum at
\[
\lambda:=\lambda_{k,m} = \argmin f(x)[(0,m)]=\sqrt{\frac{m(m+1)}{k^2+1}},
\]
for which
\begin{equation}\label{lambdakm}
f(\lambda) = \sqrt{(k^2+1)(m^2+m)} - \frac{(2m+1)k}{2} - \frac12.
\end{equation}
	Since $f$ is convex,
	\begin{align*}
	\ell_{k,m}:=\argmin f(x)[[m]]\in\{\lfloor\lambda\rfloor, \lceil\lambda\rceil\}
	\end{align*}

We may now generalize Theorem~\ref{23}(a) to arbitrary $k\ge1$. 
 \begin{prop}\label{23k}
 For all integers $k\ge1$, $r\ge0$,  and $m=k\ell+r$, if $r\le \ell<r(r+1)$, then $\bar\eta_{k,m}\ge 1/f_{k,m}(\ell)$.
\end{prop}

Note that $r\le \ell<r(r+1)$ if and only if $\frac m{k+1}\le\ell<(m-k\ell)(m-k\ell+1)$. It turns out (see Fact~\ref{AD} in Appendix) that for $m$ sufficiently large,
\begin{equation}\label{ellkm}
\frac m{k+1}\le\ell_{k,m}<(m-k\ell_{k,m})(m-k\ell_{k,m}+1),
\end{equation}
that is, the assumptions of Proposition~\ref{23k} are satisfied for $\ell=\ell_{k,m}$ and, consequently,
 $\bar\eta_{k,m}\ge 1/f_{k,m}(\ell)$.
Unlike for $k=1$, in general case we have fallen short of showing a~complementary 0-statement. Instead, we are only able to show the following, quite tight estimates.

\begin{theorem}\label{new_bounds}
For all integers $k\ge2$, there exists a constant $m_k$ such that for all $m\ge m_k$,
$$\frac1{f(\ell_{k,m})}\le \bar\eta_{k,m}\le\frac1{f(\lambda_{k,m})}.$$
\end{theorem}

The difference $\frac1{f(\lambda_{k,m})}-\frac1{f(\ell_{k,m})}$ is shown in Appendix to be less than $128k^5/m^3$, so it converges to 0 with $m\to\infty$ (see Fact~\ref{diff} therein.)

\begin{rem}
In fact, for some rather rare instances, Theorem~\ref{new_bounds} does determine $\bar\eta_{k,m}$ exactly. Namely, when $\lambda_{k,m}$ happens to be an integer, so that $\ell_{k,m}=\lambda_{k,m}$, and, at the same time, \eqref{ellkm} hold. We will show now that for any $k$ there are infinitely many values of $m$ for which this happens. Let $\lambda_{k,m} = p$, where $p$ is an integer. Equivalently, $\frac{m^2+m}{k^2+1}=p^2$ and hence, the quadratic equation of $m$,
\[
m^2+m - (k^2+1)p^2=0,
\]
has a positive integer solution if $\Delta = 1+4(k^2+1)p^2$ is a perfect square, in which case $m=\frac{-1+\sqrt{\Delta}}{2}$.
Thus, we want to know for which integers $p$ and $q$, we have
\[
1+4(k^2+1)p^2 = q^2.
\]
The latter is a well-known Pell's equation~(see,~e.g.,~\cite{Pell}). This equation has infinitely many distinct integer solutions $(p,q)$ (as already proved by Lagrange). For instance, such solutions exist for $k=2$, $m=4$ (corresponding to $p=2$ and $q=9$) and for $k=3$, $m=9$ (corresponding to $p=3$ and $q=19$). Thus, $\lambda_{2,4}=2$ and $\lambda_{3,9}=3$.
\end{rem}

Since Theorem~\ref{new_bounds} applies only to large values of $m$, we complement it by determining the missing exponents and over-exponents for $k=2$ and all $m\le 20$ (except for $m=19$), as well as for $k=3$ and $m\le20$. Recall that so far the value of $\eta_{2,m}$ has been known only for $m\in\{3,4,5,6,7\}\cup\{9,14\}$, while $\eta_{3,m}$ -- for $m\in\{4,\dots,10, 12,13, 16, 20\}$,   and that no over-exponents $\bar\eta_{2,m}$ or $\bar\eta_{3,m}$ have been discovered before. 

\begin{prop}\label{k2} We have 
$$\eta_{2,11}=\frac25,\quad \eta_{2,13}=\frac13,\quad\eta_{2,16}=\frac27,\quad\eta_{2,18}=\frac14,\quad\eta_{2,20}=\frac29\,;$$
$$\bar\eta_{2,8}=\frac12,\quad \bar\eta_{2,10}=\frac49,\quad \bar\eta_{2,12}=\frac5{13},\quad\bar\eta_{2,15}=\frac27,\quad\bar\eta_{2,17}=\frac{7}{27}\,;$$
$$ \eta_{3,15}=\frac25,\quad \eta_{3,18}=\frac13,\quad \eta_{3,19}=\frac13\,;$$
and 
$$\bar\eta_{3,11}=\frac12,\quad \bar\eta_{3,14}=\frac49,\quad\bar\eta_{3,17}=\frac5{13}.$$
\end{prop}

All known values of exponents and over-exponents, are summarized in Table~\ref{table:exponents2} for $k=2$ and  in Table~\ref{table:exponents3} for $k=3$.

\begin{table}
    \renewcommand{\arraystretch}{1.3}
    \begin{tabular}{ | >{\columncolor{codelightgray}}c || c | c |c | c | c | c | c | c | c | c | c | c | c | c | c |}
	\hline
	\rowcolor{codelightgray}
	$m$ & 3--5 &   6--7 &      $8^*$ & 9 &                  $10^*$ & $11^*$ & $12^*$ &    $13^*$ & $14$ & $15^*$ & $16^*$ & $17^*$ & $18^*$ & 19 & $20^*$ \\                \hline\hline
	$1/\eta_{2,m}$ &  1 & $\frac32$ & - & 2 & - &  $\frac52$ & - & 3 & 3 & - & $\frac72$ & - & 4 & ? & $\frac92$ \\ \hline\hline
	$1/\bar\eta_{2,m}$ & - & - & 2 & - & $\frac{9}{4}$ & - & $\frac{13}5$ & - &  -  & $\frac72$ & - & $\frac{27}7$ & - & ? & -\\ \hline
    \end{tabular}
    \caption{The  reciprocals of all known  $(2,m)$-Dirac exponents and over-exponents. New results are indicated by asterisk.}
    \label{table:exponents2}
\end{table}

\begin{table}
    \renewcommand{\arraystretch}{1.3}
    \begin{tabular}{ | >{\columncolor{codelightgray}}c || c | c |c | c | c | c | c | c | c | c | c | }
	\hline
	\rowcolor{codelightgray}
	$m$ & 4--7 & 8--10 & $11^*$ & 12--13 & $14^*$ & $15^*$ & 16 & $17^*$ & $18^*$ & $19^*$ & 20 \\                \hline\hline
	$1/\eta_{3,m}$ &  1 & $\frac32$ & - & 2 & - & - & $\frac52$ & - & 3 & 3 & 3 \\ \hline\hline
	$1/\bar\eta_{3,m}$ & - &  - & 2 & - & $\frac94$ & - & - &  $\frac{13}5$ &  - & - & -\\ \hline
    \end{tabular}
    \caption{The  reciprocals of all known  $(3,m)$-Dirac exponents and over-exponents. New results are indicated by asterisk.}
    \label{table:exponents3}
\end{table}

\section{Upper bound machinery}\label{mach}
	
In this section, we explain how the upper proof technique from \cite{ADRRS} and \cite{ADR} differentiates between thresholds and over-thresholds. The proofs of upper bounds in \cite{ADRRS} and \cite{ADR} were based on the standard Absorbing Method whose four main ingredients, Connecting, Reservoir, Absorbing, and Covering Lemmas, all claim the existence of certain constant length $m$-paths in $G\cup G(n,p)$, from which the ultimate $m$-th power of a Hamiltonian cycle is built. The proofs of each of these four lemmas consist of a deterministic part and a probabilistic part.

Let $P_s$ denote the $k$-path on $s$ vertices and, for a graph $F$, let $F(\ell)$ be the \emph{$\ell$-blow-up} of $F$ obtained from $F$ by replacing each vertex $v$ by an independent set $U_v$ of $\ell$ vertices (all sets $U_v$ pairwise vertex-disjoint) and replacing each edge $\{v,w\}$ by the complete bipartite graph between $U_v$ and $U_w$.


\subsection{Braids and their densities}\label{bra}

In the deterministic part, from the Dirac assumption $\delta(G)\ge \left(\tfrac k{k+1}+\eps\right)n$, one derives the existence of many copies of suitable $k$-paths $P_{(k+1)t}$ in $G$ and, based on that, many copies of the $\ell$-blow-ups $P_{(k+1)t}(\ell)$ of $P_{(k+1)t}$, also in $G$ (cf. \cite[Lemma 5.15]{ADRRS}). What is still missing from a full copy of an $m$-path is a graph consisting of $k+1$ vertex-disjoint copies of a graph $B(\ell,r,t)$ called a \emph{braid}.
In Figure~\ref{fig:decomp} we present the decomposition of an 8-path into a 3-blow-up of a 2-path (black edges) and 3 copies of the braid graph $B(3,2,2)$ (red edges). Note that here $k=2$ and $8=2\times3+2$.

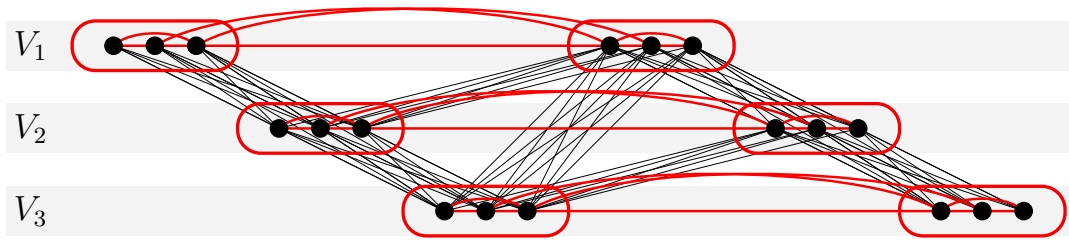
\begin{figure}[t]
\captionsetup[subfigure]{labelformat=empty}

\scalebox{1.1}
{
\begin{tikzpicture}

\coordinate (v1) at (0,0) {};
\coordinate (v2) at (0.5,0) {};
\coordinate (v3) at (1,0) {};
\coordinate (v4) at (2,-1) {};
\coordinate (v5) at (2.5,-1) {};
\coordinate (v6) at (3,-1) {};
\coordinate (v7) at (4,-2) {};
\coordinate (v8) at (4.5,-2) {};
\coordinate (v9) at (5,-2) {};
\coordinate (v10) at (6,0) {};
\coordinate (v11) at (6.5,0) {};
\coordinate (v12) at (7,0) {};
\coordinate (v13) at (8,-1) {};
\coordinate (v14) at (8.5,-1) {};
\coordinate (v15) at (9,-1) {};
\coordinate (v16) at (10,-2) {};
\coordinate (v17) at (10.5,-2) {};
\coordinate (v18) at (11,-2) {};


\foreach \i in {3,6,9,12,15} {
    \pgfmathsetmacro{\j}{\i+1}
    \draw[line width=0.1mm, color=black]  (v\i) -- (v\j);
};

\foreach \i in {2,3,5,6,8,9,1,12,14,15} {
    \pgfmathsetmacro{\j}{\i+2}
    \draw[line width=0.1mm, color=black]  (v\i) -- (v\j);
};

\foreach \i in {1,...,15} {
    \pgfmathsetmacro{\j}{\i+3}
    \draw[line width=0.1mm, color=black]  (v\i) -- (v\j);
};

\foreach \i in {1,...,14} {
    \pgfmathsetmacro{\j}{\i+4}
    \draw[line width=0.1mm, color=black]  (v\i) -- (v\j);
};

\foreach \i in {1,...,13} {
    \pgfmathsetmacro{\j}{\i+5}
    \draw[line width=0.1mm, color=black]  (v\i) -- (v\j);
};

\foreach \i in {1,...,12} {
    \pgfmathsetmacro{\j}{\i+6}
    \draw[line width=0.1mm, color=black]  (v\i) -- (v\j);
};




\foreach \i in {1,2,4,5,7,8,10,11,13,14,16,17} {
    \pgfmathsetmacro{\j}{\i+1}
    \draw[line width=0.3mm, color=red]  (v\i) -- (v\j);
};

\foreach \i in {3,6,9} {
    \pgfmathsetmacro{\j}{\i+7}
    \draw[line width=0.3mm, color=red]  (v\i) -- (v\j);
};

\draw[line width=0.3mm, color=red] (v1) .. controls (0.2, 0.2) and (0.8, 0.2) .. (v3);
\draw[line width=0.3mm, color=red] (v4) .. controls (2.2, -0.8) and (2.8, -0.8) .. (v6);
\draw[line width=0.3mm, color=red] (v7) .. controls (4.2, -1.8) and (4.8, -1.8) .. (v9);
\draw[line width=0.3mm, color=red] (v10) .. controls (6.2, 0.2) and (6.8, 0.2) .. (v12);
\draw[line width=0.3mm, color=red] (v13) .. controls (8.2, -0.8) and (8.8, -0.8) .. (v15);
\draw[line width=0.3mm, color=red] (v16) .. controls (10.2, -1.8) and (10.8, -1.8) .. (v18);

\draw[line width=0.3mm, color=red] (v2) .. controls (1.5, 0.6) and (5, 0.6) .. (v10);
\draw[line width=0.3mm, color=red] (v3) .. controls (2, 0.6) and (5.5, 0.6) .. (v11);

\draw[line width=0.3mm, color=red] (v5) .. controls (3.5, -0.4) and (7, -0.4) .. (v13);
\draw[line width=0.3mm, color=red] (v6) .. controls (4, -0.4) and (7.5, -0.4) .. (v14);

\draw[line width=0.3mm, color=red] (v8) .. controls (5.5, -1.4) and (9, -1.4) .. (v16);
\draw[line width=0.3mm, color=red] (v9) .. controls (6, -1.4) and (9.5, -1.4) .. (v17);


\foreach \i in {1,...,18} {
    \draw[fill=black, inner sep=0pt] (v\i) circle[radius=3pt] node[below] {};
};

\draw[red,line width=0.4mm, rounded corners=8pt, inner sep=0pt] (-0.5, 0.3) rectangle (1.5, -0.3) node[midway, above=10pt] {};

\draw[red,line width=0.4mm, rounded corners=8pt, inner sep=0pt] (1.5, -0.7) rectangle (3.5, -1.3) node[midway, above=10pt] {};

\draw[red,line width=0.4mm, rounded corners=8pt, inner sep=0pt] (3.5, -1.7) rectangle (5.5, -2.3) node[midway, above=10pt] {};

\draw[red,line width=0.4mm, rounded corners=8pt, inner sep=0pt] (5.5, 0.3) rectangle (7.5, -0.3) node[midway, above=10pt] {};

\draw[red,line width=0.4mm, rounded corners=8pt, inner sep=0pt] (7.5, -0.7) rectangle (9.5, -1.3) node[midway, above=10pt] {};

\draw[red,line width=0.4mm, rounded corners=8pt, inner sep=0pt] (9.5, -1.7) rectangle (11.5, -2.3) node[midway, above=10pt] {};

\foreach \i in {1,2,3} {
    \pgfmathsetmacro{\j}{\i-1};
    \node[color=black] at (-1, -\j) {$V_{\i}$};
}

\foreach \i in {-0.3,-1.3,-2.3} {
    \pgfmathsetmacro{\j}{\i+0.6}
    \fill[black,ultra nearly transparent] (-1.3,\i) rectangle (11.65,\j);
}

\end{tikzpicture}
}

\caption{The decomposition of an $8$-path.}
\label{fig:decomp}	
\end{figure}

Roughly speaking, a braid graph consists of an ordered collection of $\ell$-cliques  joined sequentially by bridge-like structures.

\newpage

\begin{dfn}\label{bra}~
\begin{enumerate}
    \item[(a)] For $r\ge1$, two sequences of vertices $\seq{v} = (v_1, v_2, \ldots, v_r)$ and $\seq{u} = (u_1, u_2, \ldots, u_r)$ of a graph $G$ form an \emph{$r$-bridge} if each $v_i$ is adjacent in $G$ to all $u_1, u_2, \ldots, u_{i}$, $i \in [r]$.
    \item[(b)]	For $t\ge1$, $\ell\ge2$, and $1\le r\le\ell$, the \emph{braid graph} $B(\ell,r,t)$ consists of $t$ vertex-disjoint $\ell$-cliques $K_\ell^{(1)},K_\ell^{(2)},\ldots,K_\ell^{(t)}$, with vertices ordered arbitrarily,  where for each $i\in[t-1]$, the last $r$ vertices of $K_\ell^{(i)}$ and the first $r$ vertices of $K_\ell^{(i+1)}$ form an $r$-bridge. We write shortly $B_t:=B(\ell,r,t)$ and, for any integer $s\geq 1$, denote by $sB_t$ the union of $s$ vertex-disjoint copies of $B_t$.
\end{enumerate}
\end{dfn}

In the probabilistic part of the proof of each lemma, based on a version of Janson's inequality, one a.a.s.~complements many of the copies of $P_{(k+1)t}(\ell)$ obtained in the deterministic part, with a suitable copy of $(k+1)B(\ell,r,t)$, coming entirely from $G(n,p)$. We now state the appropriate version of Janson's inequality used in these proofs.

For a graph $F$ with $v_F\ge1$ vertices and $e_F$ edges, set
\[d_F:=\frac{e_F}{v_F-1} \quad \text{ and } \quad m_F:=\max_{H\subseteq F}d_H.\]

\begin{lemma}[\cite{ADRRS}, Prop. 5.8]
Let $F$ be a graph with $e_F\ge1$ and, for $\tau>0$, let $\mathcal F$ be a family of copies of $F$ in $K_n$ with $|\mathcal F|\ge\tau n^{v_F}$. Furthermore, for some $C>0$, let $p\ge Cn^{-1/m_F}$ and let $X$ be the number of copies of $F\in\mathcal F$ that are present in $G(n,p)$.  Then, for some $C'$ depending on $C$ and $F$,
\[
\mathbb P(X \ge\tau n^{v_F}p^{e_F}/2)\ge1-\exp\{-C'n\}.
\]
\end{lemma}

It is easy to see that for any graph $F$, the maximum in $m_F$ is always achieved by a~connected subgraph of $F$. Thus, as $(k+1)B_t$ is a disjoint union of copies of $B_t$, we have $m_{(k+1)B_t}=m_{B_t}$.
Hence, it all boils down to calculating the maximum density $m_{B_t}$ of the braid graph. To this end, it was shown in \cite[proof of Prop. 5.8]{ADRRS} and \cite[Prop. 4.3]{ADR} that
\[
m_{B_t}= \begin{cases}
    d_{K_\ell}=\frac{\ell}2\quad\mbox{if}\quad \ell\ge r(r+1),\\ d_{B_t} = \frac{t{\ell\choose 2}+(t-1){r+1\choose 2}}{t\ell-1}\quad\mbox{otherwise.}
\end{cases}
\]
This is the only place where the upper bound proofs in \cite{ADRRS} and \cite{ADR} differ. If there exists a~representation $m=k\ell+r$ with $\ell\ge r(r+1)$, the assumption $p\ge Cn^{-2/\ell}$ suffices to carry on the whole proof and we recover Theorem~\ref{21}(a).

Under the opposite assumption, $r\le\ell<r(r+1)$, the situation is less clear-cut. The  braids we use are of an enormous length $t=t(\eps)$ coming from the Regularity Lemma. So, in this case $m_{B_t}=d_{B_t}$  varies with $t$ (and is also a function of $\eps$). This is the reason for switching in \cite{ADR}  to the limiting density of $B(\ell,r,t)$ as $t\to\infty$. Recalling the definition of function $f_{k,m}$ from previous section, we have $d_{B_t}\nearrow f_{k,m}(\ell)$,
as $d_{B_t}=f_{k,m}(\ell)-\Theta(1/t)$ is a strictly increasing function of $t$. Thus, setting $\mu(\eps)=1/d_B(t)-1/f_{k,m}(\ell)$ we obtain Proposition~\ref{23k}. Then, the left-hand side inequality in Theorem~\ref{new_bounds} follows by Fact~\ref{AD} showing that \eqref{ellkm} holds for sufficiently large $m$.

\medskip

In summary, taking the ominous upper bound proof in \cite{ADRRS} as a black box, in order to get a feeling of the nature of the threshold associated with a particular case of $k$ and $m$, it suffices to check which decompositions $m=k\ell+r$, with $\ell\ge r$,
minimize the maximum density $m_{B_t}$ of the braid $B(\ell,r,t)$ and whether $\ell\ge r(r+1)$ or not. In the former case, we are after a threshold $d_{k,m}(n)$, in the latter -- after an over-threshold $\bar d_{k,m}(n)$. The main difficulty always lies in proving a matching 0-statement.


\subsection{Ordinary thresholds are doomed to extinction?}\label{extinct}

In this section, we argue that over-thresholds are by far more expected than the ordinary ones.  But first we have to learn how to recognize which statement yields a better bound, Theorem~\ref{21}(a) or Proposition~\ref{23k}. If it is the former, chances are that we are looking at an ordinary threshold $d_{k,m}(n)$, otherwise, at an over-threshold $\bar d_{k,m}(n)$. Both, provided a matching 0-statement can be proven. 
For convenience, we restate both statements in terms of Dirac exponents and over-exponents.

\begin{cor}\label{2cases} For all integers $k\geq 1$, $r\geq 0$, $\ell\ge r$, and $m=k\ell+r$,
\begin{enumerate}
    \item if $\ell\geq r(r+1)$, then $\eta_{k, m}\ge 2/\ell$; \label{cond:thresh}
    \item if $\ell<r(r+1)$, then $\bar\eta_{k,m}\ge 1/f(\ell)$. \label{cond:over_thresh}
\end{enumerate}
\end{cor}
\noindent Note that, by monotonicity, both parts of Corollary~\ref{2cases} hold for all $m\le k\ell+r$.

We now define three crucial parameters which will be helpful in telling the two cases apart.
Given $k\ge1$ and $m>k$, let
$$r_{cr}(k,m)=\max\{r\ge0:\ (m-r)/k \mbox{ is an \emph{integer} and } (m-r)/k\ge r(r+1)\}.$$

Note that for $k\ge2$, due to non-divisibility, the parameter $r_{cr}(k,m)$ may not exist.
Specifically, it is easy to see that for $m\equiv s\pmod{k}$, $1\le s\le k-1$, $r_{cr}(k,m)$  exists if and only if $m\ge ks(s+1)+s$. Indeed, if $m\equiv s \pmod{k}$, then for every representation $m=k\ell +r$, we have $\ell\le \frac{m-s}k$ and $r\ge s$. Thus, if $m<ks(s+1)+s$, then  $\ell<r(r+1)$, so $r_{cr}(k,m)$ does not exist, while
\begin{equation}\label{mrs}
\mbox{if}\quad m\ge ks(s+1)+s,\quad\mbox{then}\quad r_{cr}(k,m)\ge s,
\end{equation}
since $s$ satisfies both conditions in the definition of $r_{cr}(k,m)$.

In particular, for $k=2$, only $r_{cr}(2,3)$ does not exist, while for $k=3$ the exceptions are $m=4$ ($s=1$) and $m= 5,8,11,14,17$ ($s=2$). Since $s\le k-1$, it follows from~\eqref{mrs} that
\begin{equation}\label{rkm}
r_{cr}(k,m)\quad\mbox{ exists for all}\quad m\ge (k^2+1)(k-1),
 \end{equation}
 though it might do for smaller $m$ as well, for instance, for all $m\equiv 0 \pmod{k}$. See the second row of Tables~\ref{table:rcr1}--\ref{table:rcr3} for the values of $r_{cr}(k,m)$ with $k=1,2,3$ and  small $m$.
	
If $r_{cr}:=r_{rc}(k,m)$ does exist, it tells us what are the best bounds coming from Corollary~\ref{2cases}. Indeed,
putting
$$\ell_{cr}:=\ell_{cr}(k,m)=\frac{m-r_{cr}}k$$
and
$$\ell^*:=\ell^*_{k,m}=\argmin f(x)[\{r_{cr},\dots,\ell_{cr}-1\}],$$
we get
\begin{equation}
\eta_{k, m}\ge 2/\ell_{cr} \quad \text{and} \quad \bar\eta_{k,m}\ge\frac1{f(\ell^*)}
\end{equation}
from cases (\ref{cond:thresh}) and (\ref{cond:over_thresh}), respectively. Thus, it boils down to checking if

\begin{equation}\label{if}
f(\ell^*)\le\ell_{cr}/2.
\end{equation}
Note that~\eqref{ellkm} is equivalent to $\ell_m<\ell_{cr}$, which implies that $f(\ell_m)\le f(\ell_{cr})$. Moreover, one can show that $f(\ell_{cr})\le\ell_{cr}/2$. Thus,~\eqref{ellkm} implies~\eqref{if}. However, the inverse implication is not true in general. For $k=1$, each $m\in\{5,6,9\}$ is a counterexample -- see Table~\ref{table:rcr1} for the relevant parameters. The advantage of~\eqref{if}  is even more transparent for $k=2$: for $m\le34$, among the fourteen values of $m$ for which $~\eqref{if}$ holds, only  two, 25 and 32, satisfy~\eqref{ellkm} (details omitted).

Whichever holds, case (\ref{cond:thresh}) or (\ref{cond:over_thresh}) of Corollary~\ref{2cases}, it prompts a strong indication of the nature of the presumable threshold, ordinary or over-threshold. Except for a handful of small examples, in all known instances this prediction has not failed. Consult Tables~\ref{table:rcr1}--\ref{table:rcr3} for the values of all relevant parameters for $k=1,2,3$ and small $m$; the circled entries determine the Dirac exponents and over-exponents, which agree with the values presented in Tables~\ref{table:exponents}--\ref{table:exponents3}.

\begin{table}
    \renewcommand{\arraystretch}{1.3}
    \begin{tabular}{ | >{\columncolor{codelightgray}}c || c | c |c | c | c | c | c | c | c | }
	\hline
	\rowcolor{codelightgray}
			$m$ & 2\, & 3\, & 4\, & 5\, & 6\, & 7\, & 8\, & 9\, & 10 \\ \hline\hline
			$r_{cr}(1,m)$ & 0 & 1 & 1 & 1 & 1 & 1 & 2 & 2 & 2  \\ \hline\hline
			$\ell_{cr}(1,m)$ & \circled{2} & \circled{2} & \circled{3} & 4 & 5 & 6 & \circled{6} & 7 & 8 \\ \hline\hline
			$\ell_{1,m}$ & 2 & 2 & 3 & 4 & 5 & 5 & 6 & 7 & 7  \\ \hline\hline
			$\ell^*_{1,m}$ & 1 & 1 & 2 & 3 & 4 & 5 & 5 & 6 & 7 \\ \hline\hline
			$f(\ell_{1,m})$ & $\frac12$ & 1 & $\frac43$ & $\frac74$ & $\frac{11}5$ & \circled{$\frac{13}5$} & 3 & $\frac{24}7$ & \circled{$\frac{27}7$} \\ \hline\hline
			$f(\ell^*_{1,m})$ & 1 & 3 & 2 & \circled{2} & \circled{$\frac94$} & \circled{$\frac{13}5$} & $\frac{16}5$ & \circled{$\frac72$} & \circled{$\frac{27}7$} \\ \hline
    \end{tabular}
    \caption{Relevant parameters for $k=1$ and $2\le m\le10$. Circled numbers determine Dirac exponents $2/\ell_{cr}$ or over-exponents $1/f(\ell^*)$ and $1/f(\ell)$.}
    \label{table:rcr1}
\end{table}

\begin{table}
    \renewcommand{\arraystretch}{1.3}
    \begin{tabular}{ | >{\columncolor{codelightgray}}c || c | c |c | c | c | c | c | c | c | c | c | c | c |c|c|}
    \hline
    \rowcolor{codelightgray}
    $m$ & 7\, & 8\, & 9\, & 10\, & 11\, & 12\, & 13\, & 14\, & 15\, & 16\,&17\,&18\,&19\,&20 \\ \hline\hline
    $r_{cr}(2,m)$ & 1 & 0 & 1 & 0 & 1 & 0 & 1 & 2 & 1 & 2 & 1 & 2 & 1 & 2  \\
    \hline\hline
    $\ell_{cr}(2,m)$ & \circled{3} & 4 & \circled{4} & 5 & \circled{5} & 6 & \circled{6} & \circled{6} & 7 & \circled{7} & 8 & \circled{8} & 9 & \circled{9} \\
    \hline\hline
    $\ell_{2,m}$ & 3 & 4 & 4 & 5 & 5 & 6 & 6 & 6 & 7 & 7 & 8 & 8 & 9 & 9 \\
    \hline\hline
    $\ell^*_{2,m}$ & 2 & 3 & 3 & 4 & 4 & 5 & 5 & 5 & 6 & 6 & 7 & 7 & 8 & 8  \\
    \hline\hline
    $f(\ell_{2,m})$ & $\frac43$ & $\frac32$ & $\frac74$ & $2$ & $\frac{11}5$ & $\frac52$ & $\frac83$ & $3$ & $\frac{22}7$ & $\frac{24}7$ & $\frac{29}8$ & $\frac{31}8$ & $\frac{37}9$& $\frac{13}3$ \\
    \hline\hline
    $f(\ell^*_{2,m})$ & $\frac72$ & \circled{2} & 3 & \circled{$\frac94$} & 3 & \circled{$\frac{13}5$} & $\frac{16}5$ & 4 & \circled{$\frac72$} & $\frac{17}4$ & \circled{$\frac{27}7$} & $\frac{31}7$ & $\frac{17}4$ & $\frac{19}4$ \\ \hline
    \end{tabular}

\caption{Relevant parameters for $k=2$ and $m\in\{7,\dots,20\}$. Circled numbers determine Dirac exponents $2/\ell_{cr}$ or over-exponents $1/f(\ell^*)$. The only open case is $m=19$.}
\label{table:rcr2}
\end{table}

\begin{table}
    \renewcommand{\arraystretch}{1.3}
    \begin{tabular}{ | >{\columncolor{codelightgray}}c || c | c | c | c | c | c | c | c | c | c |c|c|}
    \hline
    \rowcolor{codelightgray}
    $m$ &  10\, & 11\, & 12\, & 13\, & 14\, & 15\, & 16\,&17\,&18\,&19\,&20 \\ \hline\hline
    $r_{cr}(3,m)$  & 1 & - & 0 & 1 & - & 0 & 1 & - & 0 & 1 & 2  \\
    \hline\hline
    $\ell_{cr}(3,m)$  & \circled{3} & - & \circled{4} & \circled{4} & - & \circled{5} & \circled{5} & - & \circled{6} & \circled{6} & \circled{6} \\
    \hline\hline
    $\ell_{3,m}$ & 3 & 3 & 4 & 4 & 4 & 5 & 5 & 5 & 6 & 6 & 6 \\
    \hline\hline
    $\ell^*_{3,m}$  & 2 & - & 3 & 3 & - & 4 & 4 & - & 5 & 5 & 5  \\
    \hline\hline
    $f(\ell_{3,m})$ & $\frac43$ & \circled{$2$} & $\frac32$ & $\frac74$ & \circled{$\frac94$} & $2$ & $\frac{11}5$ & \circled{$\frac{13}5$} & $\frac{5}2$ & $\frac{8}3$& $3$ \\
    \hline\hline
    $f(\ell^*_{3,m})$ & $\frac{11}2$ & - & $3$ & $\frac{13}3$ & - & 3 & 4 & - & $\frac{16}5$ & $4$ & $5$  \\ \hline
    \end{tabular}

\caption{Relevant parameters for $k=3$ and $m\in\{10,\dots,20\}$. Circled numbers determine Dirac exponents $2/\ell_{cr}$ or over-exponents $1/f(\ell)$.}
\label{table:rcr3}
\end{table}

\medskip

We have already noticed (cf.~\eqref{rkm}) that for each $k$, if $m$ is large enough, then $r_{cr}(k,m)$ exists. In fact, the same is true with respect to the validity of inequality~\eqref{if}, which implies that ordinary thresholds are likely to vanish for large $m$.
%
%
\begin{restatable}[]{prop}{rcrr}
\label{rcr}
For every $k\ge1$ and $m>k$, if $r_{cr}\ge 5k/2$, then $ f(\ell_{cr}-1)\le \ell_{cr}/2$. Consequently,  if $m\ge(7k/2)^3$, then~\eqref{if} holds. Moreover, for $k=2$,~\eqref{if} holds for all odd $m\ge 15$, except $m=27,29,31,33$, and for all even $m\ge 24$, except $m=44,46$.
\end{restatable}
\noindent This proposition will be proved in Appendix.	

For $k=2$, Proposition~\ref{rcr}, together with Proposition~\ref{k2} (cf. Table~\ref{table:rcr2}), implies that the only remaining values of $m$ for which~\eqref{if} does not hold, and so, ordinary threshold are likely,  though not yet proved, are $m\in\{22,27,29,31,33,44,46\}.$
It is possible to produce such (finite) lists of $m$'s for higher $k$, but they get longer and longer. For instance, for $k=3$,~\eqref{if} holds for all $m\ge 142$, as well as for certain intervals of smaller $m$, e.g., for all $m\in\{21,\dots,36\}\cup\{69,\dots,129\}$  which are divisible by 3.

\medskip

If the parameter $r_{cr}(k,m)$ does not exist, then it could be either way, a threshold or over-threshold.
Let us look at some examples.
	For $k+1\le m\le2k-1$, $k\ge2$,  by Theorem~\ref{21}(c) with $\ell=2$ and $r=0$, we get $\eta_{k,m}(n)=1$, the same Dirac exponent as for $m=2k$ and $m=2k+1$. Similarly, for $k=7$, $m=20$,  Theorem~\ref{21}(c) with $\ell=3$ and $r=0$ yields $\eta_{7,20}=3/2$.
	In turn, for $k=3$, $m=17$,  by Theorem~\ref{21}(a) with $\ell=6$ and $r=0$, we get $\eta_{3,17}\ge 1/3$. However, Proposition~\ref{23k}, with $\ell=5$ and $r=2$, yields a better bound $\bar\eta_{3,17}\ge 5/13$. Note that $f_{3,17}(5)=13/5$, which turns out to be the correct over-exponent (cf. Proposition~\ref{k2}).

	
\section{General 0-statement}\label{lb}

In this section, we prove the right-hand side inequality in Theorem~\ref{new_bounds}. To better understand this proof, we recall the proofs of Theorems~\ref{21}(b) and~\ref{23}(b) from, respectively, \cite{ADRRS} and \cite{ADR}. In doing so, we will be relying on some elementary facts about the random graph $G(n,p)$.

 Part (c) below was stated without proof in \cite[Fact 2.1]{ADRRS} with reference to \cite[Remark 3.7]{JLR}. For completeness, we provide a short proof here. A graph with $L$ vertices and $M$ edges will be called an \emph{$(L,M)$-graph}. For a graph $F$, let $X_F$ count the number of copies of $F$ in $G(n,p)$. In particular, set $X_s:=X_{K_s}$.

\begin{fact}\label{Markov}~
\begin{enumerate}
\item[(a)] Let $L$ and $M$ be two positive integers. If $p=o(n^{-L/M})$, then a.a.s.~$G(n,p)$ contains no $(L,M)$-graph as a subgraph.
\item[(b)] Let $s\ge3$. If $p=o(n^{-2/s})$, then a.a.s.~$G(n,p)$ contains $o(n)$ copies of $K_s$.
\item[(c)] Let $s\ge3$. For every $\eps>0$ there exists $c>0$ such that if $p\le cn^{-2/s}$, then a.a.s.~$G(n,p)$ contains at most $\eps n$ copies of $K_s$.
\end{enumerate}
\end{fact}

\begin{proof}
(a) For every graph $F$ with $L$ vertices and $M$ edges, by Markov's inequality,
\[\Prob(X_F>0)\le \E X_F=O(n^Lp^M)=o(1).\]

(b) Set $\omega:=1/(pn^{2/s})$. By the assumption on $p$, we have $\omega\to\infty$ and, by Markov's inequality,
$$\Prob(X_s\ge n/\omega)\le \omega\E X_s/n=O\left(\omega^{1-\binom s2}\right)=o(1).$$

(c) By monotonicity, we assume that $p = cn^{-2/s}$, and let $c=(\eps/2)^{1/\binom s2}$. We have $\E X_s\sim c^{\binom s2}n=(\eps/2)n$, while summing over all pairs of edge-intersecting copies of $K_s$ in $K_n$, and noticing that $\E X_s=o(\E X_t)$ for all $2\le t\le s-1$, we get
\begin{align*}
    Var X_s & =O\left(\sum_{t=2}^sn^{2s-t}p^{2\binom s2-\binom t2}\right) = O\left(\sum_{t=2}^s\frac{(\E X_s)^2}{\E X_t}\right)  =O\left(\E X_s\right) =O(n).
\end{align*}
Thus, by Chebyshev's inequality,
$$\Prob\left(|X_s-\E X_s|\ge n^{2/3}\right)\le\frac{Var X_s}{n^{4/3}}=o(1).$$
\end{proof}

Next, we describe a generic construction behind all 0-statements concerning powers of Hamiltonian cycles.
\smallskip

{\bf Construction.} For integers $k\ge1$ and $n\ge k+1$, $n$ divisible by $k+1$, and for any $\eps<1/(k+1)$,
let $G^{(k)}_\varepsilon=:G_\varepsilon$ be an $n$-vertex graph obtained from the complete, balanced $(k+1)$-partite graph on $n$ vertices with partition $U_1\cup\dots\cup U_{k+1}$ as follows. For each $i=1,\dots,k+1$  fix a subset $W_i\subset U_i$ of size $|W_i|=\lceil\eps n\rceil$ and insert
an unbalanced bipartite complete graph between $W_i$ and $U_i\setminus W_i$.  Clearly, $\delta(G_\varepsilon) \geq \left(\frac{k}{k+1}+\varepsilon\right)n$. For future references, set $W=W_1\cup \cdots\cup W_{k+1}$.

\smallskip

Let us first see how this construction yields the 0-statement in Theorem~\ref{21}(b).

\begin{proof}[Proof of Theorem~\ref{21}(b)] Fix $\eps=1/(2m(k+3))$ and let $c_0=c_0(\ell,\epsilon)$ be the constant from Fact~\ref{Markov}(c).
Let $p\le c_0n^{-2/\ell}$ and suppose that $G_\eps\cup G(n,p)$ contains a Hamiltonian $m$-cycle~$C$. Since $m\ge(k+1)(\ell-1)$, by the pigeonhole principle, every clique $K_{m+1}$ in $C$, disjoint from $W$, must induce a clique $K_\ell$ in $G(n,p)$. There are at least
$$\left\lfloor\frac n{m+1}\right\rfloor-|W|\ge\frac n{2m}-(k+1)\eps n=2\eps n$$
 such cliques. Hence,
$$\Prob(G_\eps\cup G(n,p)\in\ham^m_n)\le\Prob(G(n,p)\;\mbox{contains at least $2\eps n$ cliques $K_\ell$})$$
and, by Fact~\ref{Markov}(c), the latter probability is $o(1)$. Thus, $\eta=2/\ell$ satisfies Definition~\ref{thres}(ii).
\end{proof}

To prove Theorem~\ref{23}(b) we need to state a crucial lemma \cite[Lemma 2.2]{ADR}. Recall the definition of $\ell_m$ from Section~\ref{known}.

\begin{lemma}\label{lem:pathedges}
 Let $m\geq 2$ and let $P$ be an $m$-path with $V(P)=A\cup B$, $A\cap B=\emptyset$. Then,
\[|E(P[A])| + |E(P[B])| \geq f(\ell_m)|V(P)| -2m^2.\]
\end{lemma}

\begin{rem} The original assumption in \cite[Lemma 2.2]{ADR} forbidding copies of $K_{m+1}$ inside $P[A]\cup P[B]$ was redundant. \end{rem}

\begin{proof}[Proof of Theorem~\ref{23}(b)] Fix $\mu>0$ and choose $\eps$ to be small enough compared to $\mu$.
For $k=1$, the graph $G_\eps$ consists of just two sets, $U_1$ and $U_2$. Let $p=n^{-1/f(\ell_m)-\mu}$. Suppose that $G_\eps\cup G(n,p)$ contains a Hamiltonian $m$-cycle $C$ and remove all vertices of $W$. What remains of $C$ is a collection of at most $2\eps n$ $m$-paths, among which there must be one of length at least $(n-2\eps n)/(2\eps n)>1/(3\eps)$. Let $P$ be a sub-$m$-path of this $m$-path of length exactly $L:=\lceil 1/(3\eps)\rceil$. Then, by Lemma~\ref{lem:pathedges} with $A=V(P)\cap U_1$ and $B=V(P)\cap U_2$, we conclude that a spanning subgraph of $P$ with $M:=f(\ell_m)L -2m^2$ edges is contained in $G(n,p)$. Hence,
$$\Prob(G_\eps\cup G(n,p)\in\ham^m_n)\le\Prob(G(n,p)\;\mbox{contains an $(L,M)$-subgraph}).$$
However, setting $f:=f(\ell_m)$, one can easily check that
$$\frac1f+\mu>\frac{L}{fL-2m^2},$$
equivalently, $L>2m^2/f+2\mu m^2$, for $\eps$ sufficiently small. Consequently, $p=o(n^{-L/M})$ and, by Fact~\ref{Markov}(a), $\Prob(G(n,p)\;\mbox{contains an $(L,M)$-subgraph})=o(1)$.
Thus, $\eta=1/f(\ell_m)$ satisfies Definition~\ref{over}(ii).
\end{proof}

For the proof of the right-hand side inequality in Theorem~\ref{new_bounds} all we need is an analog of Lemma~\ref{lem:pathedges} with $|E(P[A])| + |E(P[B])|$ replaced by $\sum_{i=1}^{k+1}|E(P[V_{j}])|$, where $V(P)=V_1\cup\dots\cup V_{k+1}$, and
 with $f(\ell_{k,m})$ replaced by $f(\lambda_{k,m})$. A slightly weaker lower bound has been already established by an analytical method  in \cite[Section 4]{AR}. Indeed, the function $f(k,m,L)$ defined therein, with $L:=|V(P)|$ (in \cite{AR} the authors use $n$ for $L$), equals $\sum_{i=1}^{k+1}|E(P[V_{j}])|$, and it was proved in \cite[Theorem 4.1]{AR} that
 $$f(k,m,L)\ge\left(m\sqrt{k^2+1}-mk-\frac12\right)L+O(m).$$ 
 Note that the coefficient at $L$ is, for large $m$, just a bit smaller than, but, in fact, asymptotically (as $m\to\infty$) equal to $f(\lambda_{k,m})$ (cf.~\eqref{lambdakm}).

We will prove the desired lower bound by purely combinatorial means. However, instead of bounding the quantity $\sum_{i=1}^{k+1}|E(P[V_{j}])|$,
we are going to bound the number of edges spanned by a \emph{single}, suitably selected,  subset of $V(P)$. It turns out that this will be sufficient to deduce the right-hand side inequality in Theorem~\ref{new_bounds}.

\begin{lemma}\label{weak}
For $k\ge 2$ and $m\geq k+1$ if $P$ is an $m$-path and $V_0\subset V(P)$ satisfies
$|V_0|\ge|V(P)|/(k+1)$, then
\[
    |E(P[V_{0}])|\geq f(\lambda_{k,m})|V_{0}| -  m(m+1).
\]
\end{lemma}

 It was already shown in Section~\ref{bra} how the left-hand side inequality in Theorem~\ref{new_bounds} follows from Proposition~\ref{23k}.
We conclude this section by completing the proof of Theorem~\ref{new_bounds}, assuming Lemma~\ref{weak} holds. It relies on Lemma~\ref{weak} in a similar way Theorem~\ref{23}(b) does on Lemma~\ref{lem:pathedges}.  The proof of Lemma~\ref{weak} is deferred to the next section.

\begin{proof}[Proof of Theorem~\ref{new_bounds}, the right-hand side inequality]
Fix $\mu>0$ and choose $\eps$ to be small enough compared to $\mu$.
 Let $p=n^{-1/f(\lambda_{k,m})-\mu}$. Suppose that $G_\eps^{(k)}\cup G(n,p)$ contains a~Hamiltonian $m$-cycle $C$ and remove all vertices of $W$. What remains of $C$ is a collection of at most $(k+1)\eps n$ $m$-paths, among which there must be one of length at least $(n-(k+1)\eps n)/((k+1)\eps n)>1/((k+2)\eps)$. Let $P$ be a sub-$m$-path of this $m$-path of length exactly $L:=\lceil 1/((k+2)\eps)\rceil$ and let $j\in[k+1]$ be such that $|V(P)\cap U_j|\ge L/(k+1)$. Then, by Lemma~\ref{weak} with $V_0=V(P)\cap U_j$, we infer that $|E(P[V_{0}])|\geq f(\lambda_{k,m})L_0 - m(m+1)$ with $L_0=|V_0|$. It follows that
 a subgraph of $L_0$ vertices and $M_0:=f(\lambda_{k,m})L_0- m(m+1)$ edges is contained in $G(n,p)$, so
$$\Prob(G_\eps\cup G(n,p)\in\ham^m_n)\le\Prob(G(n,p)\;\mbox{contains an $(L_0,M_0)$-subgraph}).$$
  However, setting $f:=f(\lambda_{k,m})$ and observing that $L_0\ge 1/((k+2)(k+1)\eps)$,
  we have
  $$\frac1f+\mu>\frac{L_0}{fL_0-m(m+1)},$$
  for $\eps$ sufficiently small.
   Consequently, $p=o(n^{-L_0/M_0})$ and, by Fact~\ref{Markov}(a),
  $$\Prob(G(n,p)\;\mbox{contains an $(L_0,M_0)$-subgraph})=o(1).$$
 Thus, $\eta=1/f(\lambda_{k,m})$ satisfies Definition~\ref{over}(ii).
\end{proof}

\section{Algorithm REWIRE and proof of Lemma~\ref{weak}}

For the proof of Lemma~\ref{weak}, we will first modify the path $P$ by changing the order of its vertices suitably. To this end, we will construct an algorithm REWIRE (its pseudocode can be found later in this section) and show some properties of its output (Lemma~\ref{perform} below).

For an $m$-path $P$, let $\overrightarrow{V(P)}$ be a linear ordering of the vertices in $V(P)$ in which $P$ traverses them - there are two such orders and we pick one arbitrarily.
Given $P$ and a subset $V_0\subset V(P)$, let us define a sequence of nonempty segments of $\overrightarrow{V(P)}$ as follows. Assume $\overrightarrow{V(P)}=(v_1,\dots,v_L)$, where $L=|V(P)|$. For each $i\in[L]$, let $b_i=0$ if $v_i\in V_0$ and $b_i=1$ otherwise. Let us define two sequences of segments  in $\overrightarrow{V(P)}$:
$S_1, S_2,\dots,S_q$ -- corresponding to the 0-runs of the binary sequence $(b_1,\dots,b_L)$ and $T_0, T_1, T_2,\dots, T_q$ -- corresponding to the 1-runs, both listed in the order they appear in $\overrightarrow{V(P)}$. If $b_1=0$ ($b_L=0$), we artificially define $T_0=\emptyset$ ($T_q=\emptyset$), although sets $T_0$ and $T_q$ are irrelevant for counting the edges of $P[V_0]$. In particular, for all~$i$, $S_i \subset V_{0}$ and  $T_i \cap V_{0} = \emptyset$. Thus, we obtain a partition of $\overrightarrow{V(P)}$ into segments, namely $\overrightarrow{V(P)} = T_0\cup S_1\cup T_1\cup S_2\cup T_2 \cup \cdots\cup S_q\cup T_q$. We will call such a partition \emph{$V_0$-based}.  See Figure~\ref{fig:alg1} for an example.

Algorithm REWIRE keeps changing the order of vertices in $P$ to achieve constraints on the sizes $|S_i|$  and $|T_i|$, $i\in[q]$, as formulated in the next lemma, while not letting the number of edges in $P[V_0]$ grow.

\begin{figure}[t]
\captionsetup[subfigure]{labelformat=empty}

\scalebox{1.1}
{
\begin{tikzpicture}

\coordinate (v1) at (0,0) {};
\coordinate (v2) at (1,0) {};
\coordinate (v3) at (2,-1) {};
\coordinate (v4) at (3,-1) {};
\coordinate (v5) at (4,-1) {};
\coordinate (v6) at (5,-1) {};
\coordinate (v7) at (6,-1) {};
\coordinate (v8) at (7,0) {};
\coordinate (v9) at (8,-1) {};
\coordinate (v10) at (9,-1) {};
\coordinate (v11) at (10,0) {};
\coordinate (v12) at (11,0) {};

\draw[line width=0.3mm, color=black] (v1) -- (v3);
\draw[line width=0.3mm, color=black] (v1) -- (v4);

\draw[line width=0.3mm, color=black] (v2) -- (v4);
\draw[line width=0.3mm, color=black] (v2) -- (v5);

\draw[line width=0.3mm, color=black] (v3) .. controls (3,-0.75) .. (v5);
\draw[line width=0.3mm, color=black] (v3) .. controls (3.5,-0.75) .. (v6);

\draw[line width=0.3mm, color=black] (v4) .. controls (4,-0.75) .. (v6);
\draw[line width=0.3mm, color=black] (v4) .. controls (4.5,-0.75) .. (v7);

\draw[line width=0.3mm, color=black] (v5) .. controls (5,-0.75) .. (v7);
\draw[line width=0.3mm, color=black] (v5) -- (v8);

\draw[line width=0.3mm, color=black] (v6) -- (v8);
\draw[line width=0.3mm, color=black] (v6) .. controls (6.5,-0.75) .. (v9);

\draw[line width=0.3mm, color=black] (v7) -- (v9);
\draw[line width=0.3mm, color=black] (v7) .. controls (7.5,-0.75) .. (v10);

\draw[line width=0.3mm, color=black] (v8) -- (v10);
\draw[line width=0.3mm, color=black] (v8) -- (v11);

\draw[line width=0.3mm, color=black] (v9) -- (v11);
\draw[line width=0.3mm, color=black] (v9) -- (v12);

\draw[line width=0.3mm, color=black] (v10) -- (v12);

\foreach \i in {1,...,11} {
    \pgfmathsetmacro{\j}{\i+1}
    \draw[line width=0.8mm, color=black]  (v\i) -- (v\j);
};

\foreach \i in {1,...,12} {
    \draw[fill=black, inner sep=0pt] (v\i) circle[radius=3pt] node[below] {};
};

\draw[blue,line width=0.4mm, rounded corners=8pt, inner sep=0pt] (-0.5, 0.4) rectangle (1.5, -0.4) node[midway, above=12pt] {$S_1$};
\draw[blue,line width=0.4mm, rounded corners=8pt, inner sep=0pt] (1.5, -0.6) rectangle (6.5, -1.4) node[midway, above=12pt] {$T_1$};
\draw[blue,line width=0.4mm, rounded corners=8pt, inner sep=0pt] (6.5, 0.4) rectangle (7.5, -0.4) node[midway, above=12pt] {$S_2$};
\draw[blue,line width=0.4mm, rounded corners=8pt, inner sep=0pt] (7.5, -0.6) rectangle (9.5, -1.4) node[midway, above=12pt] {$T_2$};
\draw[blue,line width=0.4mm, rounded corners=8pt, inner sep=0pt] (9.5, 0.4) rectangle (11.5, -0.4) node[midway, above=12pt] {$S_3$};

\foreach \i in {1,2,8,11,12} {
    \pgfmathsetmacro{\j}{\i-1};
    \node[color=black] at (\j,-2) {\footnotesize{$b_{\i}\!\!=\!\!0$}};
};

\foreach \i in {3,...,7,9,10} {
    \pgfmathsetmacro{\j}{\i-1};
    \node[color=black] at (\j,-2) {\footnotesize{$b_{\i}\!\!=\!\!1$}};
};

\node[color=black] at (-1.7, 0) {$V_{0}$};
\node[color=black] at (-1, -1) {$V(P)\!\setminus\!V_{0}$};

\foreach \i in {-0.3,-1.3,-2.3} {
    \pgfmathsetmacro{\j}{\i+0.6}
    \fill[black,ultra nearly transparent] (-2,\i) rectangle (11.65,\j);
}

\end{tikzpicture}
}

\caption{A 3-path on 12 vertices and a $V_0$-based partition with $V_0=\{1,2,8,11,12\}$ (the sets $T_0=\emptyset$ and $T_q=\emptyset$ not shown). The thick edges show the order in which the 3-path traverses through the vertex set.}
\label{fig:alg1}	
\end{figure}
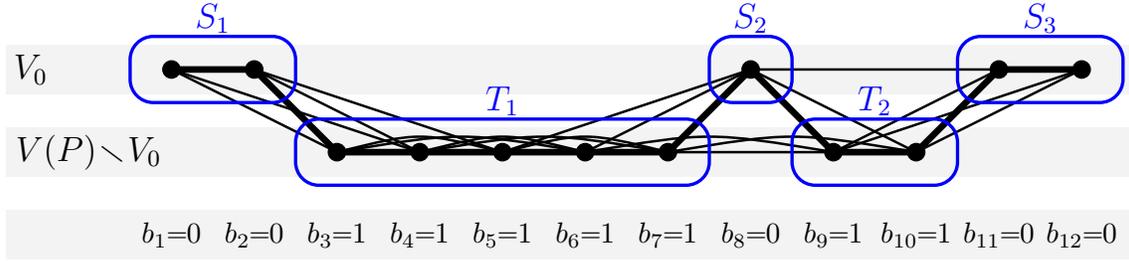	

\begin{lemma}\label{perform}
Given an $m$-path $P$ and a subset $V_0\subset V(P)$, algorithm REWIRE outputs a~modified path $P'$ with $V(P')=V(P)$, $|E(P'[V_0])|\le |E(P[V_0])|$, and a $V_0$-based partition $\overrightarrow{V(P')} = T_0\cup S_1\cup T_1\cup\cdots\cup S_q\cup T_q$ such that, setting $x_i=|S_i|$ and $y_i=|T_i|$, $i\in[q]$, we have
\begin{itemize}
    \item[(i)] $1 \leq x_i \leq m$ for $i\in[q]$,
    \item[(ii)] $0 \leq y_i \leq m$ for $i\in[q-1]$,
    \item[(iii)] $x_i + y_i \geq m$ for $i\in[q-1]$,
    \item[(iv)] $y_i + x_{i+1} \geq m$ for $i\in[q-2]$.
\end{itemize}
\end{lemma}

Note that condition (i) above guarantees that each segment $S_i$ induces in $P'[V_{j}]$ a~clique of size ${x_i\choose 2}$, while conditions (ii)--(iv), combined, guarantee that each pair $(S_i, S_{i+1})$, $i\in[q-2]$, induces in $P'[V_{j}]$ a bridge of size exactly ${m+1-y_i\choose 2}$ (cf. Definition~\ref{bra}(a)).

Before proving Lemma~\ref{perform} we show how it implies Lemma~\ref{weak}.

\begin{proof}[Proof of Lemma~\ref{weak}] Let $P$ and $V_0$ be as in the assumptions of the lemma. Setting $L:=|V(P)|$ and $L_0:=|V_0|$, we have
\begin{equation}\label{eq:largest_color}
    k L_0 \ge L - L_0.
\end{equation}

After applying algorithm REWIRE to the pair $(P,V_0)$, we obtain a new path $P'$ with $|E(P[V_0])| \geq |E(P'[V_0])|$ and a $V_0$-bounded partition $\overrightarrow{V(P')}$ satisfying conditions (i)--(iv) of Lemma~\ref{perform}. Condition (iv) does not say anything about $i=q-1$, so we do not know if there is a full $m$-bridge between sets $S_{q-1}$ and $S_q$. To accommodate this deficiency, set $e(S_{q-1},S_q)$ for the number of edges in $P$ with one endpoint in $S_{q-1}$ and the other in $S_q$ and notice that
\begin{equation}\label{lastbridge}
0\le{m+1-y_{q-1}\choose 2}-e(S_{q-1},S_q)\le\binom{m+1}2.
\end{equation}

There is also an issue with condition (ii). Since  we do not have control over how large $y_q$ might be, we subdivide it into terms equal to $m$ plus some remainder, that is, we write
$y_q=y'_q+\cdots+ y'_{q'}$, where $y'_q=\cdots=y'_{q'-1}=m$ and $y'_{q'}\le m$ (when $y_q\le m$, we just have $q'=q$, so nothing changes), and drop the primes over the $y$'s, for convenience. Then ${m+1-y_{i}\choose 2}=0$ for $i=q,\dots, q'-1$, while
\begin{equation}\label{lasty}
{m+1-y_{q'}\choose 2}\le\binom{m+1}2.
\end{equation}
To equalize the number of terms in both sums below, we also, quite artificially, set $x_i:=0$ for all $q+1\le i\le q'$.

 Applying Jensen's inequality to the convex function $\phi(z) = \frac{z(z-1)}{2}$ twice,  with $z_i=x_i$ and $z_i=m+1-y_i$, $i\in[q']$,
 we get
 \begin{equation}\label{Jensen}
 \sum_{i=1}^{q'} {x_i \choose 2}\ge q'{L_0/q' \choose 2}\quad\mbox{and}\quad\sum_{i=1}^{q'} {m+1-y_i\choose 2}\geq q'{m+1-(L-L_0)/q' \choose 2}.
 \end{equation}

 Using (\ref{eq:largest_color})--\eqref{Jensen} and recalling from Section~\ref{newres} the definition of function $f:=f_{k,m}$ with minimum at $\lambda:=\lambda_{k,m}$,
we infer that
\begin{align*}
    |E(P[V_0])| \geq |E(P'[V_0])| & \geq  \sum_{i=1}^q {x_i \choose 2} + \sum_{i=1}^{q-2} {m+1-y_i\choose 2}+e(S_{q-1},S_q) \\
    &\overset{\eqref{lastbridge},\eqref{lasty}}{\geq }\sum_{i=1}^{q'} {x_i \choose 2} + \sum_{i=1}^{q'} {m+1-y_i\choose 2} - m(m+1)\\
    &\hspace{-0.15cm} \overset{\eqref{Jensen}}{\geq}  q'{L_0/q' \choose 2} + q'{m+1-(L-L_0)/q' \choose 2} -  m(m+1) \\
    &\hspace{-0.15cm} \overset{(\ref{eq:largest_color})}{\ge} q'{L_0/q'\choose 2} + q'{m+1-kL_0/q'\choose 2} -  m(m+1) \\
    & =  f(L_0/q')L_0 - m(m+1) \geq f(\lambda)L_0 - m(m+1).
\end{align*}
\end{proof}

It remains to describe the algorithm REWIRE, and prove Lemma~\ref{perform}.
In what follows, by a slight abuse of notation we will treat segments as both sets and sequences of vertices, depending on the context.


As mentioned earlier, to achieve its goal, REWIRE keeps changing the order of vertices in $P$. It proceeds by incremental, local changes, from left to right, never going back to what has been already handled. Every change of the order of vertices results in a new $V_0$-based partition into sets $S_i$ and $T_i$. In each step, a new partition is obtained from the previous one by just a single shift of some part of segment $S_i$ or $T_i$ to, resp., $S_{i+1}$ or $T_{i+1}$, or vice versa.
We found it more transparent to describe these steps directly in terms of such shifts.

 For this purpose REWIRE uses operations $\textrm{SHIFT\_RIGHT}(A,B,h)$ and \newline $\textrm{SHIFT\_LEFT}(A,B,h)$, where $A$ and $B$ are two segments of the vertex set $\overrightarrow{V(P)}$, $A$ to the left of $B$, which are separated by some segment $C$, and $h$ is a positive integer. Given $h$, let $A_1, A_2, B_1, B_2$ be such that $A=A_1A_2$ and $B=B_1B_2$, with $|A_2|=\min\{h,|A|\}$ and $|B_1|=\min\{h,|B|\}$.

 $\textrm{SHIFT\_RIGHT}(A,B,h)$ changes the order of vertices $ACB$ to $A_1CA_2B$ and resets $A = A\setminus A_2$ and $B = B\cup A_2$, while $\textrm{SHIFT\_LEFT}(A,B,h)$ changes the order of $ABC$ to $AB_1CB_2$ and resets $A = A\cup B_1$ and $B = B\setminus B_1$. In other words, $\textrm{SHIFT\_RIGHT}(A,B,h)$ moves the last $\min\{h,|A|\}$ vertices of $A$ to the front of $B$, while $\textrm{SHIFT\_LEFT}(A,B,h)$ moves the first $\min\{h,|B|\}$ vertices of $B$ to the end of $A$. 

The four steps of the outer loop of the algorithm correspond to the four conditions (i)--(iv) in Lemma~\ref{perform}. We illustrate them by Figures~\ref{fig:step1}--\ref{fig:step4}. In all these figures, the thick black line indicates the order  $\overrightarrow{V(P')}$, while the original order is always from left to right.
Finally, note that throughout the iterations the number of sets in the partition may both decrease or increase.

\begin{algorithm}\label{alg:REWIRE}
\setstretch{0.8}
\SetKwComment{Comment}{}{}
\SetKwFunction{shiftright}{SHIFT\_RIGHT}
\SetKwFunction{shiftleft}{SHIFT\_LEFT}
\SetKwProg{myproc}{procedure}{}{}
\SetKw{And}{\textbf{and}}
\caption{REWIRE}\label{alg:rewire}

{\small{

\KwIn{\justifying A pair $(P,V_0)$, where $P$ is an $m$-path on vertex set $V(P)$ and $V_0\subset V(P)$; let
$\overrightarrow{V(P)}=T_0\cup S_1\cup T_1\cup\dots\cup S_q\cup T_q$,
where  $T_0$ and $T_q$ might be empty, be the $V_0$-based partition of $V(P)$.}
\medskip

\KwOut{\justifying A path $P'$ with $V(P')=V(P)$ such that $|E(P'[V_0])|\le |E(P[V_0])|$ and the $V_0$-based partition $\overrightarrow{V(P')}=T'_0\cup S'_1\cup T'_1\cup\dots\cup S_q'\cup T_q'$ satisfies conditions (i)--(iv) of Lemma~\ref{perform}.}
\medskip

\Comment{/* In the description of the algorithm below we suppress superscripts~$'$.\ */}
\medskip

$i = 1$\;
\While{$i\leq q$}{
\Indp
    \hspace{-0.5cm}{\footnotesize{\texttt{(1)}}}\! \If{$|S_i|>m$}{
        \If{$i=q$}{
            $S_{q+1}:=\emptyset$\;
            $T_{q+1}:=\emptyset$\;
            $q:= q+1$\;
        }
        \shiftright{$S_i,S_{i+1},|S_i|-m$}\;
    }
    \hspace{-0.5cm}{\footnotesize{\texttt{(2)}}}\! \If{$|T_i|>m$ \And $i<q$}{
        \shiftright{$T_i,T_{i+1},|T_i|-m$}\;
    }
    \hspace{-0.5cm}{\footnotesize{\texttt{(3)}}}\! \If{$|S_i|+|T_i|<m$ \And $i<q$}{
        \shiftleft{$S_i,S_{i+1},m-|S_i|-|T_i|$}\;
    \If{$S_{i+1}=\emptyset$}{
        $T_i=T_{i}\cup T_{i+1}$\;
        \For{$j\ge i+1$}{
            $S_{j}:=S_{j+1}$\;
            $T_j:=T_{j+1}$\;
        }
        $q:=q-1$\;
        \If{$|T_i|>m$ \And $i<q$}{
            \textbf{go to} {\footnotesize{\texttt{(2)}}}\;
        }
        \If{$i<q$}{
            \textbf{go to} {\footnotesize{\texttt{(3)}}}\;
        }
    }
    }

    \hspace{-0.5cm}{\footnotesize{\texttt{(4)}}}\! \If{$|T_i|+|S_{i+1}|<m$ \And $i<q-1$}{
        \shiftleft{$T_i,T_{i+1},m-|T_i|-|S_{i+1}|$}\;
        \If{$T_{i+1}=\emptyset$}{
            $S_{i+1}=S_{i+1}\cup S_{i+2}$\;
            $T_{i+1}:=T_{i+2}$\;
            \For{$j\ge i+2$}{
                $S_{j}:=S_{j+1}$\;
                $T_j:=T_{j+1}$\;
                }
            $q:=q-1$\;
            \textbf{go to} {\footnotesize{\texttt{(4)}}}\;
        }
   }

   $i = i+1$\;
}
}}
\end{algorithm}

\begin{proof}[Proof of Lemma~\ref{perform}] First, observe that the algorithm comes to a stop eventually. Indeed, in each iteration of the outer ``while'' loop, the internal loop \texttt{(2)--(3)--(2)} can be performed only finitely many times, since each time the value of $x_i + y_i$ increases until it exceeds $m$, or $q$ decreases until $i$ reaches $q$. For loops \texttt{(3)--(3)} and \texttt{(4)--(4)}, the situation is analogous. Second, the algorithm outputs a sequence $(T_0,S_1,\dots,S_{q},T_{q})$ which satisfies conditions (i)--(iv) corresponding to steps \texttt{(1)--(4)}.

To complete the proof, we now argue that in each step of the algorithm, with respect to the edges spanned by vertices in $V_0$, the number of edges destroyed is at least as large as the number of edges created, and so, ultimately, $|E(P'[V_0])|\le |E(P[V_0])|$.

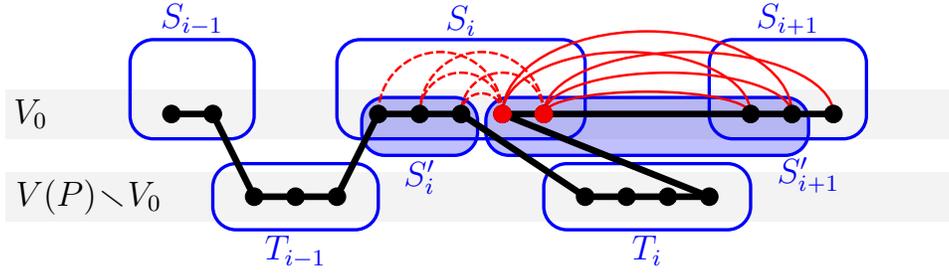
\begin{figure}[t]
\captionsetup[subfigure]{labelformat=empty}

\scalebox{1.1}
{
\begin{tikzpicture}

\coordinate (v1) at (0,0) {};
\coordinate (v2) at (0.5,0) {};
\coordinate (v3) at (1,-1) {};
\coordinate (v4) at (1.5,-1) {};
\coordinate (v5) at (2,-1) {};
\coordinate (v6) at (2.5,0) {};
\coordinate (v7) at (3,0) {};
\coordinate (v8) at (3.5,0) {};
\coordinate (v9) at (4,0) {};
\coordinate (v10) at (4.5,0) {};
\coordinate (v11) at (5,-1) {};
\coordinate (v12) at (5.5,-1) {};
\coordinate (v13) at (6,-1) {};
\coordinate (v14) at (6.5,-1) {};
\coordinate (v15) at (7,0) {};
\coordinate (v16) at (7.5,0) {};
\coordinate (v17) at (8,0) {};

\draw[blue,line width=0.4mm, rounded corners=8pt, inner sep=0pt] (-0.5, 0.9) rectangle (1, -0.3) node[midway, above=18pt] {$S_{i-1}$};
\draw[blue,line width=0.4mm, rounded corners=8pt, inner sep=0pt] (0.5, -1.4) rectangle (2.5, -0.6) node[midway, below=13pt] {$T_{i-1}$};
\draw[blue,line width=0.4mm, rounded corners=8pt, inner sep=0pt] (2, 0.9) rectangle (5, -0.3) node[midway, above=18pt] {$S_i$};
\draw[blue,line width=0.4mm, rounded corners=8pt, inner sep=0pt] (4.5, -1.4) rectangle (7, -0.6) node[midway, below=13pt] {$T_i$};
\draw[blue,line width=0.4mm, rounded corners=8pt, inner sep=0pt] (6.5, 0.9) rectangle (8.4, -0.3) node[midway, above=18pt] {$S_{i+1}$};
\draw[blue,line width=0.4mm, rounded corners=8pt, inner sep=0pt] (2.3, 0.2) rectangle (3.7, -0.5) node[midway, below=11pt] {$S_i'$};
\fill[blue,nearly transparent,rounded corners=8pt, inner sep=0pt]  (2.3, 0.2) rectangle (3.7, -0.5) {};
\draw[blue,line width=0.4mm, rounded corners=8pt, inner sep=0pt] (3.8, 0.2) rectangle (7.7, -0.5) node[right, below=0pt] {$S_{i+1}'$};
\fill[blue,nearly transparent,rounded corners=8pt, inner sep=0pt] (3.8, 0.2) rectangle (7.7, -0.5) {};

\foreach \i in {1,...,7,9,11,12,13,15,16} {
    \pgfmathsetmacro{\j}{\i+1}
    \draw[line width=0.8mm, color=black]  (v\i) -- (v\j);
};
\draw[line width=0.8mm, color=black]  (v8) -- (v11);
\draw[line width=0.8mm, color=black]  (v9) -- (v14);
\draw[line width=0.8mm, color=black]  (v10) -- (v15);

\draw[line width=0.3mm, color=red, dash pattern=on 3pt off 1pt] (v9) arc (0:180:0.25);
\draw[line width=0.3mm, color=red, dash pattern=on 3pt off 1pt] (v9) arc (0:180:0.5);
\draw[line width=0.3mm, color=red, dash pattern=on 3pt off 1pt] (v9) arc (0:180:0.75);
\draw[line width=0.3mm, color=red, dash pattern=on 3pt off 1pt] (v10) arc (0:180:0.5);
\draw[line width=0.3mm, color=red, dash pattern=on 3pt off 1pt] (v10) arc (0:180:0.75);

\draw[line width=0.3mm, color=red] (v15) arc (0:180:1.25 and 0.27);
\draw[line width=0.3mm, color=red] (v15) arc (0:180:1.5 and 0.75);
\draw[line width=0.3mm, color=red] (v16) arc (0:180:1.5 and 0.5);
\draw[line width=0.3mm, color=red] (v17) arc (0:180:1.75 and 0.75);
\draw[line width=0.3mm, color=red] (v16) arc (0:180:1.75 and 1);

\foreach \i in {1,...,17} {
    \draw[fill=black, inner sep=0pt] (v\i) circle[radius=3pt] node[below] {};
};
\draw[red,fill=red, inner sep=0pt] (v9) circle[radius=3pt] node[below] {};
\draw[red,fill=red, inner sep=0pt] (v10) circle[radius=3pt] node[below] {};

\node[color=black] at (-1.7, 0) {$V_{0}$};
\node[color=black] at (-1, -1) {$V(P)\!\setminus\!V_{0}$};

\foreach \i in {-0.3,-1.3} {
    \pgfmathsetmacro{\j}{\i+0.6}
    \fill[black,ultra nearly transparent] (-2,\i) rectangle (9.5,\j);
}

\end{tikzpicture}
}

\caption{Illustration to step (1) of algorithm REWIRE. Operation $\textrm{SHIFT\_RIGHT}(S_i,S_{i+1},2)$ moves the last two vertices of $S_i$ (red) to $S_{i+1}$.
Some of the original edges (dashed red) inside $S_i$ disappear. Instead, at least as many new edges (red) are added.}
\label{fig:step1}	
\end{figure}	

Step \texttt{(1)}: It suffices to compare how many edges are lost/gained by moving the last vertex $v$ of $S_i$ to the front of $S_{i+1}$. Before the shift, $v$ had $m$ neighbors in $S_i$ (because $x_i>m$) and some number $z$ of neighbors to the right. If $y_i\ge m$, then $z=0$ and after the shift $v$ has no neighbors to the left and at most $m$ neighbors to the right, so the balance is non-positive. If $y_i<m$, then $z\le m-y_i$ and after the shift, $v$ has $m-y_i$ neighbors to the left and at most $z+y_i$ neighbors to the right, thus again, the balance is non-positive (see Figure~\ref{fig:step1}).

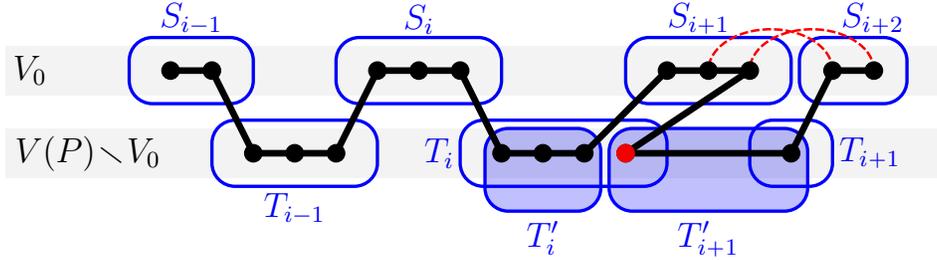
\begin{figure}[t]
\captionsetup[subfigure]{labelformat=empty}

\scalebox{1.1}
{
\begin{tikzpicture}

\coordinate (v1) at (0,0) {};
\coordinate (v2) at (0.5,0) {};
\coordinate (v3) at (1,-1) {};
\coordinate (v4) at (1.5,-1) {};
\coordinate (v5) at (2,-1) {};
\coordinate (v6) at (2.5,0) {};
\coordinate (v7) at (3,0) {};
\coordinate (v8) at (3.5,0) {};
\coordinate (v9) at (4,-1) {};
\coordinate (v10) at (4.5,-1) {};
\coordinate (v11) at (5,-1) {};
\coordinate (v12) at (5.5,-1) {};
\coordinate (v13) at (6,0) {};
\coordinate (v14) at (6.5,0) {};
\coordinate (v15) at (7,0) {};
\coordinate (v16) at (7.5,-1) {};
\coordinate (v17) at (8,0) {};
\coordinate (v18) at (8.5,0) {};

\draw[blue,line width=0.4mm, rounded corners=8pt, inner sep=0pt] (-0.5, 0.4) rectangle (1, -0.4) node[midway, above=12pt] {$S_{i-1}$};
\draw[blue,line width=0.4mm, rounded corners=8pt, inner sep=0pt] (0.5, -1.4) rectangle (2.5, -0.6) node[midway, below=13pt] {$T_{i-1}$};
\draw[blue,line width=0.4mm, rounded corners=8pt, inner sep=0pt] (2, 0.4) rectangle (4, -0.4) node[midway, above=12pt] {$S_{i}$};
\draw[blue,line width=0.4mm, rounded corners=8pt, inner sep=0pt] (3.5, -1.4) rectangle (6, -0.6) node[left=37pt,midway] {$T_{i}$};
\draw[blue,line width=0.4mm, rounded corners=8pt, inner sep=0pt] (5.5, 0.4) rectangle (7.5, -0.4) node[midway, above=12pt] {$S_{i+1}\;\;$};
\draw[blue,line width=0.4mm, rounded corners=8pt, inner sep=0pt] (7, -1.4) rectangle (8, -0.6) node[right=16pt,midway] {$T_{i+1}$};
\draw[blue,line width=0.4mm, rounded corners=8pt, inner sep=0pt] (7.6, 0.4) rectangle (8.9, -0.4) node[midway, above=12pt] {$\;\;\;\;S_{i+2}$};

\draw[blue,line width=0.4mm, rounded corners=8pt, inner sep=0pt] (3.8, -0.7) rectangle (5.2, -1.7) node[below=17pt,midway] {$T_{i}'$};
\fill[blue,nearly transparent,rounded corners=8pt, inner sep=0pt]  (3.8, -0.7) rectangle (5.2, -1.7) {};
\draw[blue,line width=0.4mm, rounded corners=8pt, inner sep=0pt] (5.3, -0.7) rectangle (7.7, -1.7) node[below=17pt,midway] {$T_{i+1}'$};
\fill[blue,nearly transparent,rounded corners=8pt, inner sep=0pt]  (5.3, -0.7) rectangle (7.7, -1.7) {};

\foreach \i in {1,...,10,13,14,16,17} {
    \pgfmathsetmacro{\j}{\i+1}
    \draw[line width=0.8mm, color=black]  (v\i) -- (v\j);
};
\draw[line width=0.8mm, color=black]  (v11) -- (v13);
\draw[line width=0.8mm, color=black]  (v12) -- (v15);
\draw[line width=0.8mm, color=black]  (v12) -- (v16);

\draw[line width=0.3mm, color=red, dash pattern=on 3pt off 1pt] (v17) arc (0:180:0.75 and 0.5);
\draw[line width=0.3mm, color=red, dash pattern=on 3pt off 1pt] (v18) arc (0:180:0.75 and 0.5);

\foreach \i in {1,...,18} {
    \draw[fill=black, inner sep=0pt] (v\i) circle[radius=3pt] node[below] {};
};
\draw[red,fill=red, inner sep=0pt] (v12) circle[radius=3pt] node[below] {};

\node[color=black] at (-1.7, 0) {$V_{0}$};
\node[color=black] at (-1, -1) {$V(P)\!\setminus\!V_{0}$};

\foreach \i in {-0.3,-1.3} {
    \pgfmathsetmacro{\j}{\i+0.6}
    \fill[black,ultra nearly transparent] (-2,\i) rectangle (9.4,\j);
}

\end{tikzpicture}
}

\caption{Illustration to step (2) of algorithm REWIRE. Operation $\textrm{SHIFT\_RIGHT}(T_{i},T_{i+1},1)$  moves the last vertex (red) of $T_{i}$ to $T_{i+1}$. Some of the original bridge-edges (dashed red) between $S_{i+1}$ and $S_{i+2}$ disappear, and no new edges are created.
}
\label{fig:alg2}
\end{figure}	

Step \texttt{(2)}: Note that after the shift the segment $S_{i+1}$ is at distance exactly $m+1$ from the segment $S_{i}$, hence there are no edges going to the left from $S_{i+1}$, and the distance between $S_{i+1}$ and $S_{i+2}$ (if there is such a segment) increases, so we can only lose edges (see Figure~\ref{fig:alg2}).

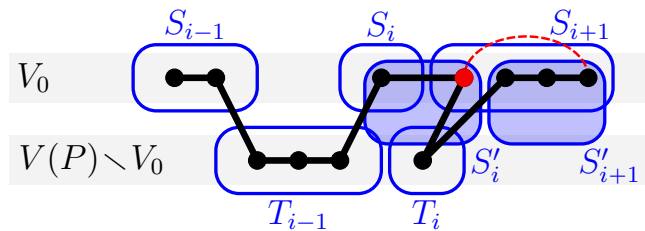
\begin{figure}[t]
\captionsetup[subfigure]{labelformat=empty}

\scalebox{1.1}
{
\begin{tikzpicture}

\coordinate (v1) at (0,0) {};
\coordinate (v2) at (0.5,0) {};
\coordinate (v3) at (1,-1) {};
\coordinate (v4) at (1.5,-1) {};
\coordinate (v5) at (2,-1) {};
\coordinate (v6) at (2.5,0) {};
\coordinate (v7) at (3,-1) {};
\coordinate (v8) at (3.5,0) {};
\coordinate (v9) at (4,0) {};
\coordinate (v10) at (4.5,0) {};
\coordinate (v11) at (5,0) {};

\draw[blue,line width=0.4mm, rounded corners=8pt, inner sep=0pt] (-0.5, 0.4) rectangle (1, -0.4) node[midway, above=12pt] {$S_{i-1}$};
\draw[blue,line width=0.4mm, rounded corners=8pt, inner sep=0pt] (0.5, -1.4) rectangle (2.5, -0.6) node[midway, below=13pt] {$T_{i-1}$};
\draw[blue,line width=0.4mm, rounded corners=8pt, inner sep=0pt] (2, 0.4) rectangle (3, -0.4) node[midway, above=12pt] {$S_i$};
\draw[blue,line width=0.4mm, rounded corners=8pt, inner sep=0pt] (2.6, -1.4) rectangle (3.5, -0.6) node[midway, below=13pt] {$T_i$};
\draw[blue,line width=0.4mm, rounded corners=8pt, inner sep=0pt] (3.1, 0.4) rectangle (5.3, -0.4) node[midway, above=12pt] {$\;\;\;\;\;\;\;\;\;\;\;\;S_{i+1}$};
\draw[blue,line width=0.4mm, rounded corners=8pt, inner sep=0pt] (2.3, 0.2) rectangle (3.7, -0.8) node[midway, below=15pt] {$\;\;\;\;\;\;\;\;\;\;\;\;\;S_i'$};
\fill[blue,nearly transparent,rounded corners=8pt, inner sep=0pt]  (2.3, 0.2) rectangle (3.7, -0.8) {};
\draw[blue,line width=0.4mm, rounded corners=8pt, inner sep=0pt] (3.8, 0.2) rectangle (5.2, -0.8) node[midway, below=15pt] {$\;\;\;\;\;\;\;\;\;\;\;\;\;S_{i+1}'$};
\fill[blue,nearly transparent,rounded corners=8pt, inner sep=0pt] (3.8, 0.2) rectangle (5.2, -0.8) {};

\foreach \i in {1,...,5,9,10} {
    \pgfmathsetmacro{\j}{\i+1}
    \draw[line width=0.8mm, color=black]  (v\i) -- (v\j);
};
\draw[line width=0.8mm, color=black]  (v6) -- (v8);
\draw[line width=0.8mm, color=black]  (v7) -- (v8);
\draw[line width=0.8mm, color=black]  (v7) -- (v9);

\draw[line width=0.3mm, color=red, dash pattern=on 3pt off 1pt] (v11) arc (0:180:0.75 and 0.5);

\foreach \i in {1,...,11} {
    \draw[fill=black, inner sep=0pt] (v\i) circle[radius=3pt] node[below] {};
};
\draw[red,fill=red, inner sep=0pt] (v8) circle[radius=3pt] node[below] {};

\node[color=black] at (-1.7, 0) {$V_{0}$};
\node[color=black] at (-1, -1) {$V(P)\!\setminus\!V_{0}$};

\foreach \i in {-0.3,-1.3} {
    \pgfmathsetmacro{\j}{\i+0.6}
    \fill[black,ultra nearly transparent] (-2,\i) rectangle (5.8,\j);
}

\end{tikzpicture}
}

\caption{Illustration to step (3) of algorithm REWIRE. Operation $\textrm{SHIFT\_LEFT}(S_i,S_{i+1},1)$  moves the first vertex (red) of $S_{i+1}$ to $S_{i}$.
One of the original edges (dashed red) in $S_{i+1}$ disappears and no new edges are added.}
\label{fig:step3}	
\end{figure}	

Step \texttt{(3)}: Let $S\subset S_{i+1}$ be the shifted subsegment. Note that each vertex of $S$ was originally at distance at most $m$ from each vertex in $S_i$. Hence, before the shift, the pair $(S_i,S)$ induced a complete bipartite graph in $P'$. Moreover (for $i\ge2$), since $y_{i-1}+x_i\geq m$, after the shift, the subsegment $S$ lies too far away from $S_{i-1}$ to create any new edges whatsoever. On the other hand, the distance between $S$ and $S_{i+1}\!\setminus\! S$ increases, so we can only lose edges (see Figure~\ref{fig:step3}).
    	
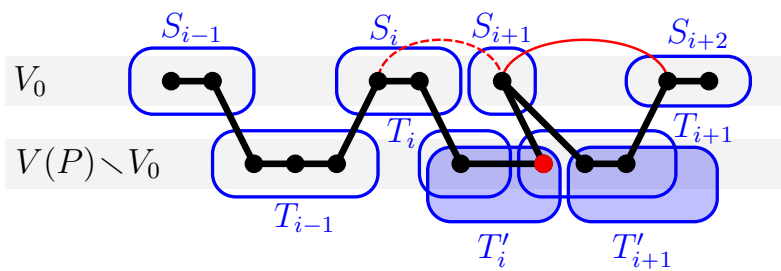
\begin{figure}[t]
\captionsetup[subfigure]{labelformat=empty}

\scalebox{1.1}
{
\begin{tikzpicture}

\coordinate (v1) at (0,0) {};
\coordinate (v2) at (0.5,0) {};
\coordinate (v3) at (1,-1) {};
\coordinate (v4) at (1.5,-1) {};
\coordinate (v5) at (2,-1) {};
\coordinate (v6) at (2.5,0) {};
\coordinate (v7) at (3,0) {};
\coordinate (v8) at (3.5,-1) {};
\coordinate (v9) at (4,0) {};
\coordinate (v10) at (4.5,-1) {};
\coordinate (v11) at (5,-1) {};
\coordinate (v12) at (5.5,-1) {};
\coordinate (v13) at (6,0) {};
\coordinate (v14) at (6.5,0) {};
\coordinate (v18) at (8.5,0) {};

\draw[blue,line width=0.4mm, rounded corners=8pt, inner sep=0pt] (-0.5, 0.4) rectangle (1, -0.4) node[midway, above=12pt] {$S_{i-1}$};
\draw[blue,line width=0.4mm, rounded corners=8pt, inner sep=0pt] (0.5, -1.4) rectangle (2.5, -0.6) node[midway, below=13pt] {$\;\;T_{i-1}$};
\draw[blue,line width=0.4mm, rounded corners=8pt, inner sep=0pt] (2, 0.4) rectangle (3.5, -0.4) node[midway, above=12pt] {$\!\!\!\!\!S_{i}$};
\draw[blue,line width=0.4mm, rounded corners=8pt, inner sep=0pt] (3, -1.4) rectangle (4.1, -0.6) node[midway, above=5pt] {$T_{i}\;\;\;\;\;\;\;\;\;\;\;\;\;$};
\draw[blue,line width=0.4mm, rounded corners=8pt, inner sep=0pt] (3.6, 0.4) rectangle (4.4, -0.4) node[midway, above=12pt] {$S_{i+1}$};
\draw[blue,line width=0.4mm, rounded corners=8pt, inner sep=0pt] (4.2, -1.4) rectangle (6.1, -0.6) node[midway, above=6pt] {$\;\;\;\;\;\;\;\;\;\;\;\;\;\;\;\;\;\;\;\;\;\;T_{i+1}$};

\draw[blue,line width=0.4mm, rounded corners=8pt, inner sep=0pt] (5.5, 0.3) rectangle (7, -0.3) node[midway, above=10pt] {$\;\;S_{i+2}$};

\draw[blue,line width=0.4mm, rounded corners=8pt, inner sep=0pt] (3.1, -1.7) rectangle (4.7, -0.8) node[below=16pt,midway] {$T_{i}'$};
\fill[blue,nearly transparent,rounded corners=8pt, inner sep=0pt]  (3.1, -1.7) rectangle (4.7, -0.8) {};
\draw[blue,line width=0.4mm, rounded corners=8pt, inner sep=0pt] (4.8, -1.7) rectangle (6.6, -0.8) node[below=16pt,midway] {$T_{i+1}'$};
\fill[blue,nearly transparent,rounded corners=8pt, inner sep=0pt]  (4.8, -1.7) rectangle (6.6, -0.8) {};
\foreach \i in {1,...,7,9,11,12,13} {
    \pgfmathsetmacro{\j}{\i+1}
    \draw[line width=0.8mm, color=black]  (v\i) -- (v\j);
};
\draw[line width=0.8mm, color=black]  (v8) -- (v10);
\draw[line width=0.8mm, color=black]  (v9) -- (v11);

\draw[line width=0.3mm, color=red] (v13) arc (0:180:1 and 0.5);
\draw[line width=0.3mm, color=red, dash pattern=on 3pt off 1pt] (v9) arc (0:180:0.75 and 0.5);

\foreach \i in {1,...,14} {
    \draw[fill=black, inner sep=0pt] (v\i) circle[radius=3pt] node[below] {};
};
\draw[red,fill=red, inner sep=0pt] (v10) circle[radius=3pt] node[below] {};

\node[color=black] at (-1.7, 0) {$V_{0}$};
\node[color=black] at (-1, -1) {$V(P)\!\setminus\!V_{0}$};

\foreach \i in {-0.3,-1.3} {
    \pgfmathsetmacro{\j}{\i+0.6}
    \fill[black,ultra nearly transparent] (-2,\i) rectangle (7.5,\j);
}

\end{tikzpicture}
}

\caption{Illustration to step (4) of algorithm REWIRE. Operation  $\textrm{SHIFT\_LEFT}(T_{i},T_{i+1},1)$ moves the first vertex (red) of $T_{i+1}$ to $T_{i}$. One of the original bridge-edges (dashed red) between $S_{i}$ and $S_{i+1}$ disappears and one bridge-edge (red) between $S_{i+1}$ and $S_{i+2}$ is created.
}
\label{fig:step4}
\end{figure}	

Step \texttt{(4)}: In this step, only the edges incident to vertices in $S_{i+1}$ are affected. Note that each vertex in $S_{i+1}$ loses exactly $t=\min(m-y_i-x_{i+1},|T_{i+1}|)$ edges going to $S_i$ (since $x_i+y_i \ge m$), and gains at most $t$ edges going to the right, so the balance is again non-positive (see Figure~\ref{fig:step4}).
\end{proof}


\section{Proof of Proposition~\ref{k2}}\label{small}

In this section we prove Proposition~\ref{k2}.
In each case, the 1-statement comes from the general approach shown in Section~\ref{mach} and the relevant value of the threshold/over-threshold (depending on the validity of inequality~\eqref{if}) is given in Table~\ref{table:rcr2} or Table~\ref{table:rcr3}. Thus, one only needs to prove the 0-statement, for which ad hoc arguments, similar to those used earlier in \cite{ADRRS} and \cite{ADR}, will have to be invented.


\subsection{General outline}


All proofs follow essentially the same line,  similar, in fact, to that used in Section~\ref{lb} to prove Theorems~\ref{23}(b) and~\ref{new_bounds} (the right-hand side inequality). For both, $k=2$ and $k=3$, in all instances we use  the graph
 $G^{(k)}_\eps$ defined in Section~\ref{lb}. Let us recall that
$G^{(k)}_\eps$ is a balanced, $(k+1)$-partite, complete $n$-vertex graph on $U_1\cup\cdots\cup U_{k+1}$, with additional $k+1$ unbalanced bipartite complete graphs inserted within each of the partition sets $U_i$, between fixed sets $W_i\subset U_i$ and their complements $U_i\!\setminus\! W_i$, where $|W_i|=\lceil\eps n\rceil$, $i\in[k+1]$.  

Given $m$, $p$, and $k\in\{2,3\}$, let $\eps=\eps(m,k)$ be such that
\begin{equation*}
    \eps\le \frac1{(k+3) (m+1)^2}.
\end{equation*}
Next, set $H=G^{(k)}_\eps\cup G(n,p)$ and let $s$ be the smallest integer such that $\Prob(X_s>\eps n)=o(1)$, where $X_s$ is the number of copies of $K_s$ in $G(n,p)$. The goal is to show that \newline $\Prob(H\in\ham^m_n,\;X_s\le\eps n)=o(1)$.

Suppose that $H$ contains  the $m$-th power of a Hamilton cycle and fix one such cycle~$C$. Further, assume that $G(n,p)$ contains at most $\eps n$ copies of $K_s$.
Next, order the vertices of $H$ according to the cyclic order determined by $C$ and remove from $H$ all vertices belonging to $W=W_1\cup \cdots\cup W_{k+1}$, as well as one vertex from each copy of $K_s$ in $G(n,p)$. Altogether, we have removed at most $(k+2)\eps n$ vertices. What remains is a collection of vertex disjoint $m$-paths and the longest of them has at least $(n-\eps n)/((k+2)\eps n)>1/((k+3)\eps)$ vertices. We truncate it so that it has precisely
\[L:=\left\lceil \frac1{(k+3)\eps}\right\rceil\]
vertices and call it $P$. Note that by the choice of $\eps$ we have $L\ge(m+1)^2$. Finally, define $k+1$ vertex-disjoint subgraphs $P_t:=P[V_t]$, where $V_t= V(P)\cap U_t$, $t\in[k+1]$. Observe that each $P_t\subset G(n,p)$ and thus $P_t\not\supset K_s$.
Set
\[F:=P_1\cup \cdots\cup P_{k+1}\]
and $M := |E(F)|$. The density of $F$ is all what matters.
Indeed, in each instance of the proof we will show that the ratio $M/L$ is large enough, so that,  by Fact~\ref{Markov}(a),
$$\Prob(H\in\ham^m_n,\;X_s\le\eps n)\le\Prob(G(n,p)\;\mbox{contains an $(L,M)$-subgraph})=o(1)$$
and the 0-statement holds (cf. Definition~\ref{over}).


\subsection{Edge suppliers}\label{prep}

Here we gather a few simple facts that can be used to provide edges in the graph $F$ defined in the previous section.

Let us begin by noticing that every $m+1$ consecutive vertices of $P$ form a clique and that an $h$-path on $b$ vertices contains precisely $hb-\binom{h+1}2$ edges.
Slightly less trivial is the following fact whose proof is deferred  to the Appendix.

\begin{restatable}[]{fact}{hpathh}
\label{hpath}
Let $m\ge k(s-1)+h$. For each $t\in[k+1]$, if $|V_t|\ge h+1$
then  $P_t$ contains a~spanning $h$-path.
\end{restatable}

By Fact~\ref{hpath} with $h=m-k(s-1)$,
we infer that each component $P_t$ of $F$ contains a~spanning $h$-path and thus
\begin{equation}\label{eq:M_h_estimate}
M\ge hL-(k+1){h+1\choose 2}.
\end{equation}
In some cases, this is enough to apply Fact~\ref{Markov}(a) and finish the proof. If not, one must estimate the number of other edges in $F$.

For $t\in[k+1]$, a \emph{$V_t$-pair} is a pair of vertices  belonging to $V_t$.
 A $V_t$-pair $\{u,v\}$ is called \emph{$q$-far} if $u$ and $v$ are separated in the ordering $\overrightarrow{V(P)}$ by exactly $q-1$ vertices of~$V_t$. If, moreover, $\{u,v\}\in E(F)$, we call it a~\emph{$q$-far edge} (of $F$). For example, for $\overrightarrow{V(P)}=[10]$ and $V_2=\{2,3,5,7,9\}$, the pair of vertices $\{2,7\}$ is a 3-far $V_2$-pair, while, if $m\ge5$, it is also a 3-far edge.

Note that an $h$-path in $P_t$ contains only $i$-far edges for $1\le i\le h$.
We are about to formulate two results which may provide $i$-far edges of $F$ for $i>h$. For $i\ge1$, let $x_i$ denote the number of $i$-far edges of $F$.

A potential source of such edges lies in segments of $m+1$ consecutive vertices of $P$ which, as was mentioned above, span cliques in~$P$.   Our first observation, whose proof is left to the reader, follows easily by double counting.

\begin{obs}\label{ob} If every $(m+1)$-segment of $\overrightarrow{V(P)}$ contains at least $z$ edges of $F$ that are $i$-far, then
\begin{equation}\label{eq:no_of_i_far_edges}
    x_i \geq \frac{z(L-m)}{m+1-i}.
\end{equation}

\end{obs}

If this is not enough, we still have one more card in the sleeve (with the proof provided in Appendix).

\begin{restatable}[]{obs}{obtwo}
\label{ob2}
Let $m= k(s-1)+h$.
Then, for all $h<i,j\le s+h-i$,
\begin{equation*}
L-x_i\le(k+1)i+\frac{i}{s+h-i-j}\cdot x_j.
\end{equation*}
In particular,
\begin{equation}\label{eq:no_of_h+1_and_j_far_edges}
L-x_{h+1}\le(k+1)(h+1)+\frac{h+1}{s-1-j}\cdot x_j.
\end{equation}
\end{restatable}

\subsection{Details}

We examine each case separately, first taking care of the ordinary thresholds, followed by the over-thresholds. Within each group, we order the values of $m$ according to the increasing complexity of the proofs.

\medskip

{\bf I. Ordinary thresholds for $k=2$}

\medskip

In each case, to claim that $\eta$ is the desired $(2,m)$-Dirac exponent, one needs to show that $M/L > 1/\eta$. Then, with $p=O(n^{-\eta})$,   apply Fact~\ref{Markov}(a).

\medskip
	
$\mathbf{m=11}$:  $\eta_{2,11}= \frac25$

\medskip

Here, the edges of the spanning $h$-paths are enough. By Fact~\ref{Markov}(c) with $s=5$, choose $c=c(\eps)$ so that with $p\le cn^{-2/5}$, a.a.s.~$X_5\le \eps n$.  Since $h=m-2(s-1)=3$, using~\eqref{eq:M_h_estimate} we have $M\ge 3L-18>\frac52L$ (recall that $L\ge(m+1)^2$).

\medskip
	
$\mathbf{m=16}$: $\eta_{2,16}= \frac27$

\medskip
The proof is analogous to the previous one. With $s=7$ and $h=4$, we get $M\ge 4L-30>\frac72L$.

\medskip

$\mathbf{m=13}$: $\eta_{2,13}= \frac13$

\medskip

In this case, we need the edges of the spanning $h$-paths together with some $(h+1)$-far edges.
By Fact~\ref{Markov}(c) with $s=6$, choose $c = c(\eps)$ so that with $p\leq cn^{-1/3}$, a.a.s.~$X_6 \leq \eps n$. Since $h = 3$, using~\eqref{eq:M_h_estimate} we have at least $3L - 18$ edges in the spanning $h$-paths.
Consider any 14-element segment of $F$ and note that since there is no $K_6$, each such segment is split with respect to the number of vertices among the three sets $V_1,V_2,V_3$ essentially in only one way (ignoring permutations), that is, with composition $14=5+5+4$. As such, it contains two 4-far edges of $F$. Then, by~\eqref{eq:no_of_i_far_edges}, we have $x_4 \geq \frac{L-13}{5}$ and, altogether, we get the bound $M\ge 3L-18+\frac{L-13}{5}>3L$.

\medskip

$\mathbf{m=18}$: $\eta_{2,18}=\frac14$

\medskip

We use the same approach as above. It is enough to take $s=8$ and $h=4$, which gives us at least $4L-30$ edges in the spanning $4$-paths. Next, note that the only possible decompositions of an 19-element segment are $19=7+7+5=7+6+6$. Thus, each such segment contains at least four $5$-far edges and we have $x_5 \geq \frac{2(L-m)}{7}$. Therefore, $M \geq 4L-30+\frac{2(L-m)}{7} > 4L$.

\medskip
	
$\mathbf{m=20}$: $\eta_{2,20}=\frac29$

\medskip

Here, apart from the edges of the spanning $h$-paths, we need some $(h+1)$-far and $(h+2)$-far edges. Take $s=9$ and $h=4$, which again yield at least $4L-30$ edges in the spanning $4$-paths. This time we have three possible compositions of $21$, that is $21=8+8+5=8+7+6=7+7+7$, and each of them brings exactly six 5-far edges and at least three 6-far edges of $F$. Therefore, by~\eqref{eq:no_of_i_far_edges}, $x_5 \geq \frac{3(L-20)}{8}$ and $x_6 \geq \frac{L-20}{5}$, and we get the bound $M \geq 4L-30 + \frac{3(L-20)}{8} + \frac{L-20}{5} > \frac92L$.

\medskip

{\bf II. Over-thresholds for $k=2$}

\medskip

We now move on to over-thresholds. To use Fact~\ref{Markov}(a), we will again show that the ratio $M/L$ is large enough. However, owing to the $-\mu$ term in the exponent in the definition of the over-threshold, the bounds which we use can be relaxed. In particular, to claim that $\bar\eta$ is the correct $(2,m)$-Dirac over-exponent, it is enough to show that
   $ M \geq L/\bar{\eta} - b$
for some positive constant $b$. Indeed, given $\mu$ and $b$, take $p(n) \leq n^{-\bar{\eta}-\mu}$ and choose $\eps>0$, so that
 $5b\eps <\frac1{\bar\eta}-\frac1{\bar\eta+\mu}$.     
 Then,
 $$\frac ML\ge \frac1{\bar\eta}-\frac bL\ge \frac1{\bar\eta}-5b\eps>\frac1{\bar\eta+\mu},$$
 and, by Fact~\ref{Markov}(a), we get the 0-statement.

Another difference is that instead of Fact~\ref{Markov}(c), it is now sufficient to use Fact~\ref{Markov}(b) to determine $s$.

\medskip

$\mathbf{m=8}$: $\bar\eta_{2,8}= \frac12$

\medskip

In this case, we will only use the edges of the spanning $h$-paths. With $s=4$ and $h=2$, by~\eqref{eq:M_h_estimate} we have
$M\ge 2L-9$.

\medskip
	
$\mathbf{m=10}$: $\bar \eta_{2,10}=\frac49$

\medskip

Here we need the spanning $h$-paths together with $(h+1)$-far edges. Take $s=5$ and $h=2$. This brings $2L - 9$ edges in the spanning $2$-paths. Since the only possible decomposition of an 11-element segment is $11 = 4+4+3$, we get two 3-far edges in each such segment, and therefore, by~\eqref{eq:no_of_i_far_edges}, $x_3 \geq \frac{L-10}{4}$. Altogether, we get $M \geq 2L - 9 + \frac{L-10}{4} = \frac94L - 11\frac12$.

\medskip

$\mathbf{m=12}$: $\bar \eta_{2,12}=\frac5{13}$

\medskip

 This time we will need a more careful analysis, where instead of estimating the numbers of $(h+1)$-far and $(h+2)$-far edges separately, we will estimate their total number together.

Take $s=6$ and $h=2$. This gives $2L-9$ edges in the spanning $2$-paths. There are two ways to decompose a $13$-element segment: $13=5+5+3=5+4+4$. In each case we get four $3$-far edges and at least one $4$-far edge, therefore $x_3 \geq \frac{2(L-12)}{5}$ and $x_4 \geq \frac{L-12}{9}$. Together, $M \geq 2L - 9 + \frac{2(L-12)}{5} + \frac{L-12}{9} = 2\frac{23}{45}L - 15\frac{2}{15}$, not good enough, since $2\frac{23}{45} < \frac{13}{5}$.

Fortunately, there is still room for improvement, based on the optimization technique from \cite[proof of Thm. 1.5,\; $m=9$]{ADR}.
Our goal is to minimize the value $x_3+x_4$ by relating $x_3$ with $x_4$. By \eqref{eq:no_of_h+1_and_j_far_edges}, we get the bound $L-x_3\le 9+3x_4$.
Thus, we arrive at the following system of inequalities
\[
\begin{cases}
    x_3 \geq \frac25 L - \frac{24}{5} \\
    x_3 \geq L - 3x_4 - 9.
\end{cases}
\]
Multiplying the first inequality by $2/3$ and the second one by $1/3$, and then adding them together, yields
\[
x_3+x_4 \ge \frac{3}{5} L - 6\frac15.
\]
Thus, we get a sufficiently good lower bound
\[M \geq 2L-9+\frac35L-6\frac{1}{5} = \frac{13}{5}L - 15\frac{1}{5}.\]

\medskip
	
    $\mathbf{m=15}$: $\bar \eta_{2,15}= \frac27$

\medskip

We use a similar optimization as in the previous case. Take $s=7$ and $h=3$, which yields $3L-18$ edges in the spanning $3$-paths. Next, there are two ways to decompose a $16$-element segment: $16=6+5+5=6+6+4$. For each one, we get four $4$-far edges and thus $x_4 \geq \frac{L-15}{3}$.

Further, by \eqref{eq:no_of_h+1_and_j_far_edges},  $L-x_4\le 12+4x_5$, so  the following system of inequalities emerges:
\[
\begin{cases}
    x_4 \geq \frac{1}{3}L - 5 \\
    x_4 \geq L - 4x_5 - 12.
\end{cases}
\]
Multiplying the first one by $3/4$ and the second one by $1/4$, we get
\[ x_4 + x_5 \geq \frac12L - 6\frac34\]
and, consequently,
\[ M \geq 3L - 18 + \frac12L - 6\frac34 = \frac72L - 24\frac34. \]

\medskip

$\mathbf{m=17}$: $\bar \eta_{2,17}=\frac7{27}$

\medskip

Once again we use the same optimization technique, but this time involving $(h+1)$-far, $(h+2)$-far and $(h+3)$-far edges. Take $s=8$ and $h=3$ to get $3L-18$ edges in the spanning $3$-paths. Since each $18$-element segment can be split in three ways: $18=7+7+4=7+6+5=6+6+6$, in each case we get six $4$-far edges, which yields the bound $x_4 \geq \frac{3(L-17)}{7}$. As for the $5$-far and $6$-far edges, using again \eqref{eq:no_of_h+1_and_j_far_edges}, we get $L - x_4 \leq 12+2x_5$ and $L - x_4 \leq 12+4x_6$. This leads to the following system of inequalities
\[
\begin{cases}
    x_4 \geq \frac37 L - \frac{17}7 \\
    x_4 \geq L - 2x_5 - 12\\
    x_4 \geq L - 4x_6 -12.
\end{cases}
\]
Multiplying the above inequalities by $1/4$, $1/2$ and $1/4$, respectively, and then adding them together gives
\[
x_4+x_5+x_6 \ge  \frac{6}{7} L - 9\frac{17}{28}.
\]

Therefore,
\[ M \geq 3L-18 + \frac{6}{7} L - 9\frac{17}{28} = \frac{27}7L- 9\frac{17}{28}.\]

\medskip

{\bf III. Ordinary thresholds for $k=3$}

\medskip

Similarly to Case {\bf I.}, to claim that $\eta$ is the desired $(3,m)$-Dirac exponent, one suffices to show that $M/L > 1/\eta$.

\medskip

    $\mathbf{m=15}$: $\eta_{3,15}=\frac25$

\medskip

It is enough to take the spanning $h$-paths. With $s=5$ and $h=3$, by~\eqref{eq:M_h_estimate} we get $M\geq 3L - 24 > \frac52L$.

\medskip

    $\mathbf{m=18}$: $\eta_{3,18}=\frac13$

\medskip

 With $s=6$ and $h=3$, by~\eqref{eq:M_h_estimate}, we get $3L-24$ edges in spanning 3-paths. Moreover, since $19=5+5+5+4$, every 19-vertex segment contains at least three 4-far edges and, by Observation~\eqref{eq:no_of_i_far_edges} we get $x_4\ge\frac{3(L-18)}{15}$. Thus, altogether, $M>3L$.

\medskip

    $\mathbf{m=19}$: $\eta_{3,19}=\frac13$

\medskip
Since we have just proved that $\eta_{3,18}=\frac13$, while $\eta_{3,20}=\frac13$ was established in~\cite{ADRRS}, the result follows by monotonicity: $\eta_{3,18}\ge\eta_{3,19}\ge\eta_{3,20}$.

\medskip

{\bf IV. Over-thresholds for $k=3$}

\medskip

Similarly to Case {\bf II.}, to claim that $\bar{\eta}$ is the correct $(3,m)$-Dirac over-exponent, it is enough to show that
\[ M \geq L/\bar{\eta}-b\]
for some positive constant $b$.

\medskip

    $\mathbf{m=11}$: $\bar \eta_{3,11}=\frac12$

\medskip

In this case the spanning $h$-paths are enough. Taking $s=4$ and $h=2$, by~\eqref{eq:M_h_estimate} we get $M\geq 2L - 12$.

\medskip

    $\mathbf{m=14}$: $\bar \eta_{3,14}=\frac49$

\medskip

Here we need the spanning $h$-paths together with $(h+1)$-far edges. Take $s=5$ and $h=2$. This gives $2L-12$ edges in the spanning $2$-paths. Since there is only one decomposition of a $15$-element segment, namely $15=4+4+4+3$, in each such segment we get three $3$-far edges. Thus, by \eqref{eq:no_of_i_far_edges}, we have $x_3\geq \frac{L-14}{4}$. Hence, $M \geq 2L-12+\frac{L-14}{4} = \frac94L -15\frac12$

\medskip

    $\mathbf{m=17}$: $\bar \eta_{3,17}=\frac5{13}$

\medskip

We will use the optimization connecting the number of $(h+1)$- and $(h+2)$-far edges. With $s=6$ and $h=2$, one gets
$2L-12$ edges in the spanning $2$-paths.
Next, since each 18-element segment can be split either with composition $5+5+5+3$ or $5+5+4+4$, this brings six 3-far edges and, by \eqref{eq:no_of_i_far_edges}, $x_3 \geq \frac{2(L-17)}{5}$. By \eqref{eq:no_of_h+1_and_j_far_edges}, we get also the relation $L-x_3 \leq 12+3x_4$. This leads to a system of inequalities
\[
\begin{cases}
    x_3 \geq \frac{2}{5}L - \frac{17}5 \\
    x_3 \geq L - 3x_4 -12,
\end{cases}
\]
which, after multiplying the first one by $\frac23$ and the second one by $\frac13$, implies that
\[ M \geq 2L-12 + x_3 + x_4 \geq 2L-12 + \frac35L - 6\frac{4}{15} = \frac{13}5L - 18\frac{4}{15}. \]

\section{Open questions}

 $\mathbf{1}$. We believe that in Theorem~\ref{new_bounds} it is the lower bound which yields the correct value of the over-threshold.

  \begin{conj}\label{correct}
For all integers $k\ge2$, there exists a constant $m_k$ such that for all $m\ge m_k$,
$$\bar\eta_{k,m}=\frac1{f(\ell_{k,m})}.$$
\end{conj}

 One way to confirm it would be to  strengthen the bound in Lemma~\ref{weak} accordingly.

\begin{conj}\label{con:pathedges}
Fix $k\ge 2$ and $m\geq k+1$. There exists a positive constant $c=c(m)$ such that if $P$ is an $m$-path with the vertex set partition $V(P)=V_1\cup\dots\cup V_{k+1}$, then there exists $j\in[k+1]$ such that
\[
|E(P[V_{j}])|\geq f(\ell_{k,m})|V_{j}| - c(m).
\]
\end{conj}

One possible line of attack at Conjecture~\ref{con:pathedges} could be to alter the algorithm REWIRE so that the output partition satisfied $y_i\le kx_i$ for each $i$. Then, we would have
$$\sum_{i=1}^{q'} {x_i \choose 2} + \sum_{i=1}^{q'} {m+1-y_i\choose 2}\ge\sum_{i=1}^{q'}x_if(x_i)\ge f(\ell_{k,m})L_0.$$

\medskip

 $\mathbf{2}$. It seems that using our techniques one can compute Dirac thresholds for more small instances. For $k=2$, in the smallest open case, $m=19$, a similar strategy to that in Section~\ref{small} fails. Hence, any 0-statement proof would require a more sophisticated approach, along the proof for ${m=12, 15, 17}$, but probably much harder. Also, one could try to determine the ordinary thresholds for $m\in\{22,27,29,31,33,44,46\}$, that is, for all values of $m$ which were singled out as the remaining likely cases in Proposition~\ref{rcr}.

 For $k=3$, within our reach are ordinary thresholds for $m=22,23,26,29$ (just based on Fact~\ref{hpath}) and $m=25,28,32,39$ (Fact~\ref{hpath} plus Observation~\ref{ob}). More ambitious seems to be  the task of determining the over-thresholds for $m=21,24, 27,30$  and the ordinary threshold for $m=31$.

 We have not tried larger values of $k$, but it  looks plausible that for some initial values of $m$ our approach should work for them as well.

\medskip

 $\mathbf{3}$. Another challenging problem is, for a given $m$ and $k$, to determine the thresholds for $\ham_n^m$ in the case where $\eps=0$, that is, when the deterministic graph $G$ satisfies only the assumption $\delta(G) \geq \frac{k}{k+1}n$, or, more generally, when $\delta(G)\ge\alpha n$ for any fixed $\alpha>0$. So far, this problem has been solved only for $m=1$ in \cite{BFM2003} and for $m=2$ in \cite{BPSS2022}.

\appendix

\section{Remaining proofs}\label{appendixA}

Here, we present the proofs of Proposition~\ref{rcr}, Facts~\ref{AD} and~\ref{diff}, as well as Fact~\ref{hpath} and Observation~\ref{ob2}. For convenience, we restate the statements below.

\rcrr*
\begin{proof} To simplify the notation, within this proof we abbreviate $r:=r_{cr}$ and $\ell:=\ell_{cr}.$ We have $m=k(\ell-1)+r+k$, so setting $\ell'=\ell-1$ and $r'=r+k$, by the definition of $r$ we get
$$\ell'<r'(r'+1)=(r+k)(r+k+1).$$
The desired inequality
$$f(\ell')=\frac{\binom{\ell'}2+\binom{r+k+1}2}{\ell'}\le \ell/2$$
or, equivalently,
\begin{equation}\label{rkl}
    (r+k+1)(r+k)\le 2\ell'
\end{equation}
follows, owing to $\ell\ge r(r+1)$, from inequality
$$(r+k+1)(r+k)\le 2\left(r(r+1)-1\right),$$
which is equivalent to
\begin{equation}\label{rk}
r^2+r\ge 2rk+k^2+k+2.
\end{equation}

Being a quadratic inequality in $r$, \eqref{rk} can be solved precisely, but we are contented with the observation that it is true for $r\ge 5k/2$. Indeed, for $k=1$ it becomes $r^2\ge r+4$, which holds for $r\ge3$. For $k\ge2$,
$$r^2=(r-k)^2+2rk-k^2\ge\frac54k^2+2rk$$
and so the left-hand side of~\eqref{rk} is at least
$$2kr+\tfrac54k^2+\tfrac52k\ge2kr+k^2+k+2,$$
as $\tfrac14k^2+\tfrac32k\ge 2$.
	
For the second part of Proposition~\ref{rcr}, first note that for $m\ge(7k/2)^3>(k^2+1)(k-1)$, the parameter $r:=r_{cr}$ exists. We now combine the expression $m=k\ell'+r+k$ with the inequality $\ell'\le(r+k)(r+k+1)-1$, to obtain the bound
$$m\le k(r+k)(r+k+1)+r\le (r+k)^3.$$
Thus, for $m\ge(7k/2)^3$, we get $r\ge5k/2$ and, by the first part of Proposition~\ref{rcr},
$$f(\ell^*)\le f(\ell')\le \ell/2.$$

Finally, let us consider the case $k=2$ and two subcases with respect to the parity of~$m$ (since $r_{cr}$ must have the same parity as $m$). For even $m$, by~\eqref{rk}, $r_{cr}\ge6$ implies~\eqref{rkl}. Thus,  by~\eqref{mrs},  $m\ge 2\cdot 6\cdot 7+6=90$ suffices. For odd $m$, the bound is even lower: we just need $r_{cr}\ge5$ and, consequently, $m\ge 2\cdot5\cdot6+5=65$.

Both of these bounds can be lowered even further by means of monotonicity.
Indeed, we have $r_{cr}(2,48)=4$ and $\ell_{cr}(2,35)=22$, and for these values inequality \eqref{rkl} holds. So it does for every greater value of $\ell_{cr}$ as long as $r_{cr}=4$ (since $\ell$ appears only on the right-hand side of \eqref{rkl}), that is, up to $m=88$. Similarly, $r_{cr}(2,24)=2$  and $\ell_{cr}(2,24)=11$, so, again, \eqref{rkl} holds for all even $m\in\{24,\dots,42\}$.

For $m$ odd we argue in the same fashion: we have $r_{cr}(2,35)=3$ and $\ell_{cr}(2,35)=16$, so for these values inequality \eqref{rkl} holds. Consequently, it does for every greater value of $\ell_{cr}$ as long as $r_{cr}=3$, that is, up to $m=63$. By the same token, since $r_{cr}(2,15)=1$  and $\ell_{cr}(2,15)=7$, inequality \eqref{rkl} holds for all $m\in\{15,\dots,25\}$.

\end{proof}

\begin{fact}
\label{AD}
For $m\ge 30k^3$, we have $\frac m{k+1}\le\ell_{k,m}<(m-k\ell_{k,m})(m-k\ell_{k,m}+1)$.
\end{fact}

\begin{proof}
For the first inequality it suffices to take $m\ge k^2$ and show that $\lambda_{k,m}\ge \frac{m+k}{k+1}$ as then $$\ell_{k,m}\ge\lfloor\lambda_{k,m}\rfloor\ge \left\lfloor\frac{m+k}{k+1}\right\rfloor\ge\frac m{k+1}.$$
Recalling that $\lambda_{k,m}=\sqrt{\frac{m^2+m}{k^2+1}}$, the inequality $\lambda_{k,m}\ge \frac{m+k}{k+1} $ is equivalent (after squaring both sides) to
\[
(m-k^2)(k^2+2km+1) \ge 0,
\]
which holds for $m\ge k^2$.

For the second inequality, first note that for $k\ge 1$,
\[
k^2+1 \ge k^2 + \frac{1}{2} + \frac{1}{16k^2} = \left(k+\frac{1}{4k} \right)^2
\]
and
\[
m^2+m \le m^2+m+\frac{1}{4} = \left(m+\frac{1}{2} \right)^2.
\]
Thus,
\[
\ell_{k,m} \le \lceil \lambda_{k,m} \rceil
\le \left\lceil \frac{m+\frac{1}{2}}{k+\frac{1}{4k}} \right\rceil
\le \frac{m+\frac{1}{2}}{k+\frac{1}{4k}} + 1=:x.
\]
Consequently, $\ell_{k,m}<(m-k\ell_{k,m})(m-k\ell_{k,m}+1)$ will follow from $x<(m-kx)(m-kx+1)$, which after some easy but tedious calculations, is equivalent to
\[
16k^6 - 12k^4 + (-24m - 12)k^3 - 9k^2 + (-6m - 3)k + m^2 + m - 1 > 0.
\]
It is not difficult to see that the latter holds for $m\ge 30k^3$, since
\begin{align*}
\underbrace{16k^6 - 12k^4}_{\ge 4k^6} + \underbrace{(-24m - 12)k^3}_{=-24mk^3-12k^3} \underbrace{- 9k^2}_{\ge-9k^3} + \underbrace{(-6m - 3)k}_{\ge-6mk^3-3k^3} + \underbrace{m^2 + m}_{30mk^3 + 30k^3}-1
\ge 4k^6+6k^3-1>0.
\end{align*}
\end{proof}	
	
Next, we provide an estimate on the difference between the bounds in Theorem~\ref{new_bounds}.

\begin{fact}
\label{diff}
For any integer $k$ and $m>k$ we have
$$\frac1{f(\lambda_{k,m})}-\frac1{f(\ell_{k,m})}\le\frac{128k^5}{m^3}.$$
\end{fact}
\begin{proof}
Recall that $f:=f_{k,m}$ is a continuous, convex function with a unique global minimum at $\lambda:=\lambda_{k,m}$ and a minimum over integer domain at $\ell:=\ell_{k,m}$. Thus,
\[
f(\ell)\le\min\{f(\lambda-1),f(\lambda+1)\}.
\]
But,
\[
 f(\lambda-1)-f(\lambda) = \frac{(k^2+1)^2}{-2k^2+2\sqrt{(k^2+1)m(m+1)}-2}
\]
\[
f(\lambda+1)-f(\lambda) = \frac{(k^2+1)^2}{2k^2+2\sqrt{(k^2+1)m(m+1)}+2},
\]
hence, we have
$$f(\ell)-f(\lambda)\le \frac{(k^2+1)^2}{2k^2+2\sqrt{(k^2+1)m(m+1)}+2}\le\frac{(k^2+1)^2}{2km}.$$
Moreover, since clearly $k^2+1 \ge \left(k+\frac{1}{4k}\right)^2$ and $m^2+m\ge \left(m+\frac{1}{4}\right)^2$, we get
\begin{align*}
f(\lambda)&=\sqrt{(k^2+1)(m^2+m)} - \frac{(2m+1)k}{2} - \frac{1}{2}\\
&\ge \left(k+\frac{1}{4k}\right)\left(m+\frac{1}{4}\right) - \frac{(2m+1)k}{2} - \frac{1}{2}
=\frac{4m-4k^2-8k+1}{16k}.
\end{align*}
Now for $m\ge 6k^2$,
\[
4m-4k^2-8k+1 = (2m + 2m) -4k^2-8k+1 \ge 2m + 12k^2 -4k^2-8k+1 \ge 2m
\]
and $f(\lambda)\ge m/(8k)$. Hence,
\[
\frac1{f(\lambda)}-\frac1{f(\ell)}\le\frac{f(\ell)-f(\lambda)}{f(\lambda)^2}\le
\frac{(k^2+1)^2}{2km} \cdot \frac{64k^2}{m^2}
= \frac{32k(k^2+1)^2}{m^3} \le \frac{32k(2k^2)^2}{m^3}
=\frac{128k^5}{m^3}.
\]
\end{proof}

\hpathh*
\begin{proof} By symmetry it suffices to prove the statement for $t=1$ only. Let $u_1,\dots,u_{h+1}$ be consecutive (in the order established by $P$) vertices of $P_1$. We claim that there are at most $k(s-1)$ other vertices between $u_1$ and $u_{h+1}$ which will yield the presence of an edge between $u_1$ and $u_{h+1}$, completing the proof.

Suppose this is not true. Then, by the pigeonhole principle, there are $s$ vertices between $u_1$ and $u_{h+1}$, all from the same set $V_t$, for some $t\in\{2,\dots,k+1\}$.   Let $u_1<v_1<\cdots <v_s<u_{h+1}$ be an $s$-tuple of such vertices such that the pair $\{v_1,v_s\}$ is separated in $\overrightarrow{V(P)}$ by the \emph{smallest possible} number of vertices. Let there be exactly $q-1$ vertices between $v_1$ and $v_{s}$. Then, to avoid $K_s$ in $F$, we must have $q-1\ge m$. Thus, there are at least $m - (h-1) \geq k(s-1)+1$ vertices of $V_2\cup\cdots\cup V_{k+1}$ between $v_1$ and $v_{s}$. Among them, again by the pigeonhole principle, there are  $s$ vertices which all belong to the same subset $V_{t'}$, for some $t'\in\{2,\dots,k+1\}$. This contradicts the choice of the $s$-tuple $\{v_1,\dots,v_s\}$, as the new $s$-tuple is squeezed between $v_1$ and $v_{s}$.
\end{proof}

\obtwo*
\begin{proof}
For $i\ge1$ and $t\in[k+1]$, let $x_i^{(t)}$ be the numbers of $i$-far edges in $P_t$.
Note that there are exactly $|V_1|-i$ $V_1$-pairs that are $i$-far, and $|V_1|-i-x_i(V_1)$ of them are \emph{not} $V_1$-edges. Let $u<v$ be one such pair, and let $x_j^{(t)}(u,v)$ be the number of $j$-far $V_t$-edges between $u$ and $v$.
The only reason for $u<v$ not to be an $i$-far edge is that $v-u\ge m+1$, that is, there are at least $m$ vertices of $P$ between $u$ and $v$. Since we are after a lower bound on $\sum_{t=2}^{k+1}x_j^{(t)}(u,v)$, we assume that there are exactly $m$ such vertices.

Among these vertices, exactly $i-1$ belong to $V_1$, while the remaining vertices are shared between the other $k$ sets. Since there are no copies of $K_s$ in any $P_t$, at most $s-1$ of these vertices belong to each set $V_t$, and so, at least $m-(i-1)-(k-1)(s-1)=s+h-i$ are in each set.
Thus, owing to the upper bound on $j$,
 $$\sum_{t=2}^{k+1}x_j^{(t)}(u,v)\ge k\left(s+h-i-j\right).$$

Let ${\mathcal P}_1$ be the set of all $i$-far $V_1$-pairs that are not edges of $F$ and denote by ${\mathcal E}_{\neq 1}$ the set of all $j$-far $V_t$-edges in $F$ for some $t\ge2$.
Imagine an auxiliary bipartite graph $\mathcal G$ between the sets of pairs ${\mathcal P}_1$ and ${\mathcal E}_{\neq 1}$ where an edge connects $(u,v)\in {\mathcal P}_1$ with $(w,z)\in {\mathcal E}_{\neq1}$ whenever $u<w<z<v$. Set $e:=|E(\mathcal G)|$ and observe that by the above argument
$$e\ge|{\mathcal P}_A|\cdot k\left(s+h-i-j\right).$$
On the other hand, the degree of each vertex $(w,z)\in {\mathcal E}_{\neq i}$ is at most $i$. Indeed, consider all pairs $(u,v)$ connected with $(z,w)$ in $\mathcal G$. Their left ends lie to the left of $w$, while their right ends to the right of $z$ (we assume $w<z$ here) and each pair is $i$-far. This implies that there cannot be more than $i$ of them.
(The maximum value of $i$ is achieved by a configuration
consisting of a block of $i$ vertices of $V_1$ followed by at least $m-i+1$ vertices of $\bigcup_{t=2}^{k+1} V_t$, followed by another block of $i$ vertices of $V_1$). Consequently, $e\le |{\mathcal E}_{\neq1}|\cdot i$ and, put together,
  $$|{\mathcal P}_1|\cdot k\left(s+h-i-j\right)\le|{\mathcal E}_{\neq1}|\cdot i,$$ or, equivalently,
  $$|V_1|-i-x_i^{(j)}\le \frac {i\sum_{t=2}^{k+1} x_i^{(t)}}{k\left(s+h-i-j\right)} .$$
  By symmetry, we also get
  $$|V_t|-i-x_i^{(t)}\le \frac{i\sum_{t'=2}^{k+1} x_i^{(t')}}{k\left(s+h-i-j\right)}$$
  for all $t=2,\dots,k+1$. Summing these $k+1$ inequalities, we get
 $$L-(k+1)i-x_i\le \frac {i}{s+h-i-j}\cdot x_j,$$	
 as required.
 \end{proof}

\end{document}